\documentclass{aims}
\usepackage{amsmath}
\usepackage{paralist}
\usepackage{graphics} 
\usepackage{epsfig} 
\usepackage{graphicx}  
\usepackage{epstopdf}
\usepackage[colorlinks=true]{hyperref}
\hypersetup{urlcolor=blue, citecolor=red}

\def\currentvolume{X}
 \def\currentissue{X}
  \def\currentyear{200X}
   \def\currentmonth{XX}
    \def\ppages{X--XX}
     \def\DOI{10.3934/xx.xx.xx.xx}

\usepackage{fullpage}
\usepackage{float}
\restylefloat{figure}

\newtheorem{theorem}{Theorem}[section]
\newtheorem{corollary}{Corollary}

\newtheorem{lemma}[theorem]{Lemma}

\theoremstyle{definition}

\newtheorem{remark}{Remark}
\newtheorem{example}{Example}
\newtheorem{assumption}{Assumption}

\title[Stochastic equations with Allee effect]{Stochastic difference equations with the Allee effect}
\author[E. Braverman and A. Rodkina]{}


\email{maelena@ucalgary.ca}
\email{alexandra.rodkina@uwimona.edu.jm}

\thanks{
E. Braverman is a corresponding author, e-mail maelena@ucalgary.ca.}

\begin{document}
\maketitle

\centerline{\scshape Elena Braverman}
\medskip
{\footnotesize
Dept. of Math. and Stats., University of Calgary,
2500 University Drive N.W., Calgary, AB, Canada T2N 1N4}

\medskip

\centerline{\scshape and Alexandra Rodkina}
\medskip
{\footnotesize
Department of Mathematics,
the University of the West Indies, Mona Campus, Kingston, Jamaica}

\medskip

\bigskip

 \centerline{(Communicated by ... )}

\begin{abstract}
For a truncated stochastically perturbed equation
$x_{n+1}=\max\{ f(x_n)+l\chi_{n+1}, 0 \}$ with $f(x)<x$ on $(0,m)$,
which corresponds to the Allee effect, we observe that for very small
perturbation amplitude $l$, the eventual behavior is similar to a non-perturbed 
case: there is extinction for small initial values in $(0,m-\varepsilon)$ and persistence
for $x_0 \in (m+\delta, H]$ for some $H$ satisfying $H>f(H)>m$.
As the amplitude grows, 
an interval $(m-\varepsilon, m+\delta)$ of initial values arises and expands, such that 
with a certain probability, $x_n$ sustains in $[m, H]$, and possibly eventually gets into the interval
$(0,m-\varepsilon)$, with a positive probability. Lower estimates for these probabilities are presented.
If $H$ is large enough, as the amplitude of perturbations grows,
the Allee effect disappears: a solution persists for any positive initial value. 
\end{abstract}

\maketitle

\section{Introduction}

Difference equations can describe population dynamics models, and, if there is no compensation 
for low population size, i.e. the stock recruitment is lower than mortality, 
the species goes to extinction, unless the initial size is large enough. This phenomenon was  introduced in 
\cite{allee_1931}, see 
also \cite{boukal,sch}. 
It is called the Allee effect after \cite{allee_1931} and can be explained by many factors: 
problems with finding a mate, deficiency of group defense or/and social functioning for low population densities. 
If the initial population size is small enough (is in the Allee zone) then the population size 
tends to zero as the time grows and tends to infinity. 
Even a small stochastic perturbation which does not tend to zero, significantly changes the situation: 
due to random immigration, there are large enough values of the population size for some 
large times even in the Allee zone, due to this occasional immigration. 
Thus, instead of extinction, we explore eventual low-density behavior, as well as 
essential persistence and solution bounds. 
Results on permanence of solutions for stochastic difference equations, including boundedness and persistence, were
recently reviewed in \cite{Sch2012}. For recent results on asymptotic behavior of stochastic difference equations also see
\cite{ABR,ABR1,AMR,BR1,BravRodk1,BravRodk2,Cohen,DR,KR,KSh,RB,Shaikhet} 
and the whole issue of Journal of Difference Equations and Applications
including \cite{Sch2012}.

The influence of stochastic perturbations on population survival, chaos control and eventual cyclic behavior was investigated in 
\cite{BravHar,BravRodk1,BravRodk2}. It was shown that the chaotic behavior could be destroyed by either a positive deterministic
\cite{BravHar} or stochastic noise with a positive mean \cite{BravRodk1,BravRodk2}; instead of chaos, there is an attractive two-cycle. 

Certainly, stochastic perturbations, applied formally, can lead to negative size values.
To avoid this situation, we consider the truncated stochastic difference equation 
\begin{equation}
\label{eq:main}
x_{n+1}=\max\bigl\{ f(x_n)+l\chi_{n+1}, 0 \bigr\}, \quad x_0>0, \quad n\in \mathbb N.
\end{equation}
Here $f$ is a function with a possible Allee zone, for example, 
\begin{equation}
\label{eq:boukal1}
x_{n+1}= \frac{Ax_n^2}{B+x_n} e^{r(1-x_n)}, 
\end{equation}
described in \cite{burgman} and
\begin{equation}
\label{eq:boukal2}
x_{n+1}= \frac{Ax_n}{B+(x_n-T)^2},
\end{equation}
considered in \cite{hop,jacobs}; see \cite{boukal} for the detailed outline
of models of the Allee effect.

It is well known that, without a stochastic perturbation, if $f(x)$ is a function such that $0<f(x)<x$ for $x \in (0,m)$ and
$f(x)>m$ for $x>m$, the eventual behavior of a solution depends on the initial condition: if $0<x_0<m$, then the solution tends
to zero (goes to extinction), if $x_0>b$ then the solution satisfies $x_n>m$, i.e. persists. 
Sometimes high densities also lead to extinction, as in \eqref{eq:boukal1} and \eqref{eq:boukal2}, we can only claim that $f(x)>m$ for $x \in (m,H)$ and conclude persistence
for $x_0 \in (m,H)$. However, the situation changes for \eqref{eq:main} with a stochastic perturbation: for example, even if $f$ has an Allee 
zone, the eventual expectation of a solution exceeds a positive number depending on $l$ and the distribution of $\chi$. Nevertheless,
this effect is due to immigration only, and we will call this type of behavior blurred extinction, or eventual low density.
In the present paper, we use some ideas developed in \cite{Brav} for models with a randomly switching perturbation.

Significant interest to discrete maps is stimulated by complicated types of behavior exhibited even by simple maps. 
In particular, for \eqref{eq:boukal1} with $r$ large enough, whatever a positive initial value is, the chaotic solution can take values in 
the interval $(0,\varepsilon)$ for any small $\varepsilon>0$. Then, in practical applications the dynamics is not in fact chaotic but leads to eventual 
extinction as the positive density cannot be arbitrarily low. Nevertheless, if the range is separated from zero, for some maps there is an unconditional survival (persistence), independently of a positive initial value.

In this note, we are mostly interested in the maps $f$ with survival for certain initial values and an Allee zone: if $x_0$ is small enough, then the solution of \eqref{eq:main} with $l=0$ tends to zero, and there is an interval $[a,H] \subset (0,\infty)$ which $f$ maps into itself. The main results of the paper are the following:
\begin{enumerate}
\item
If in \eqref{eq:main} the value of $l$ is small enough, the dynamics is similar to the non-stochastic case: blurred extinction (low density)
for small $x_0$ and persistence for $x_0$ in a certain interval.
\item
If $l>0$ is large enough then, under some additional assumptions, there is an unconditional survival.
\item
If the non-perturbed system has several attraction zones then for any initial condition, the solution can become persistent with large enough lower bound, whenever $l$ is large enough.
\end{enumerate}

The paper is organized as follows. After describing all relevant assumptions and notations in Section~\ref{sec2}, 
we state that for perturbations small enough, there is the same Allee effect as in the deterministic case, in Section~\ref{sec:per}.
The result that there may exist large enough perturbation amplitudes ensuring survival for any positive initial condition,
is also included in Section~\ref{sec:per}. Further, Section~\ref{sec4} deals with the case when, for certain initial 
conditions, both persistence and low-density behavior are possible, with a positive probability, 
while for other initial conditions, a.s. persistence
or a.s. low-density behavior is guaranteed. 
For initial values leading to different types of dynamics, lower bounds for 
probabilities of each types of dynamics
are developed in Section~\ref{sec4}. 
The case when the deterministic equation has more than 2 positive fixed point, is considered in 
Section~\ref{sec:multi}. The results are illustrated with numerical examples in Section~\ref{sec6}, and Section~\ref{sec7}
involves a short summary and discussion.


\section{Preliminaries}
\label{sec2}

Let  $(\Omega, {\mathcal{F}},  (\mathcal{F}_n)_{n \in \mathbb{N}}, {\mathbb{P}})$ be  a complete filtered probability space.  Let $\chi:=(\chi_n)_{n\in\mathbb{N}}$  be a sequence of independent random variables with the zero mean.
The filtration $(\mathcal{F}_n)_{n \in \mathbb{N}}$ is supposed to be naturally generated by 
the sequence $(\chi_n)_{n\in\mathbb{N}}$, namely 
$\mathcal{F}_{n} = \sigma \left\{\chi_{1},  \dots, \chi_{n}\right\}$.

In the paper we assume that stochastic perturbation $\chi$ in the equation \eqref{eq:main} satisfies the following assumption
\begin{assumption}
\label{as:chibound}
$(\chi_n)_{n\in \mathbb N}$ is a sequence of independent and  identically  distributed  continuous random variables with the density function 
$\phi(x)$, such that 
\[
\phi(x)>0, \quad x\in (-1, 1), \quad \phi(x)\equiv 0, \quad x\notin [-1, 1].
\]
\end{assumption}

We use the standard abbreviation ``a.s." for the wordings ``almost sure" or ``almost surely" with respect
to the fixed probability measure $\mathbb P$  throughout the text. A detailed discussion of stochastic concepts and
notation may be found in, for example, Shiryaev \cite{Shiryaev96}.

Everywhere below, for each $t\in [0, \infty)$, we denote by $[t]$ the integer part of $t$. 

Before we proceed further, let us introduce assumptions on the function $f$ in \eqref{eq:main}.

\begin{assumption}
\label{as:faH}
$f:(0, \infty)\to (0, \infty)$ is continuous, $f(0)=0$, and there exist positive numbers $a$ and  $H$, 
$a<H$,   such that
\begin{enumerate}
\item [(i)] $ f^H:=\max_{x\in [0,H]}f(x) < H$;
\item[(ii)]  $f(x)>f(a)>a, \quad x\in (a, H]$.
\end{enumerate}
\end{assumption}

So far we have not supposed that there is an Allee zone, where for small 
initial values, a solution of the non-perturbed system tends to zero. This is included in the next condition. 

\begin{assumption}
\label{as:bmain}
There is a point $b_1>0$ such that $f(x) <x$ and $f(x)\leq f(b_1)$ for $x \in (0,b_1)$.
\end{assumption}


\section{Unconditional Persistence and Low-Density Behavior}
\label{sec:per}

In this section, we consider the case when the type of perturbation and the initial condition allow us
to predict a.s. the eventual behavior of the solution. Lemma~\ref{lem:add1} indicates a small initial interval,
where the Allee effect is observed, for small enough perturbations. Lemma~\ref{lem:main0} presents the range of initial 
conditions which guarantee permanence of solutions, for $l$ small enough. However, for large enough $l$
and appropriate $f$, the Allee effect completely disappears under a stochastic perturbation, see Theorem~\ref{lem:add2}.

\begin{lemma}
\label{lem:add1}
Let Assumptions \ref{as:chibound} and \ref{as:bmain} hold, $f(b_1)<b_1$.  
Let $x_n$ be a solution of equation \eqref{eq:main} with 
\begin{equation}
\label{cond:l<}
l \leq b_1 -f(b_1)
\end{equation}
and $x_0\in [0,b_1]$. Then,   $x_n\in [0,b_1]$ for all $n \in {\mathbb N}$.
\end{lemma}
\begin{proof}
For $x_0\in [0,b_1]$, we have $f(x_0)\leq f(b_1)$ and, a.s. on $\Omega$,  
$$ 
x_1=f(x_0)+l\chi_1 \leq f(b_1)+l \leq f(b_1)+b_1-f(b_1) \leq b_1.
$$
Similarly, the induction step implies $x_n\in [0,b_1]$ for all  $n \in {\mathbb N}$, a.s.
\end{proof}

Let us introduce the function
\begin{equation}
\label{def:F}
F(x)=f(x)-x, \quad x\in [0, \infty).
\end{equation}

\begin{remark}
\label{rem:lemext}
Assumption \ref{as:bmain} holds for non-decreasing $f$ such that $f(x)<x$ for $x$ small enough.
In this case, once it is satisfied for a given $b_1>0$, this is also true for any $b \in (0,b_1)$. 
For example, if $f(x)<x$ on $(0,b_2)$ and $f(b_2)=b_2$, we can take any $b_1<b_2$ in Assumption \ref{as:bmain}. 
Then, the continuous function $F(x)=f(x)-x$ is negative on $(0,b_2)$ and vanishes at the end of the interval, so it attains its minimum at a 
point inside the interval. 
Moreover, if Assumptions \ref{as:faH} and \ref{as:bmain} hold, we have $F(a)>a$, and also there is a minimum of $F(x)$ on $[0,a]$ attained on $(0,a)$ at a 
point $b$:    
\begin{equation}
\label{def:b1}
b=\min\left\{\beta>0: F(\beta)=\min_{x\in [0,a]} F(x) \right\}.
\end{equation}
\end{remark}

\begin{lemma}
\label{lem:main0}
Let Assumptions \ref{as:chibound} and \ref{as:faH} hold, and $(x_n)$ be a solution of equation \eqref{eq:main}, 
with the noise amplitude $l$  satisfying  
\begin{equation}
\label{cond:mainl}
l < \min\{H-f^H, \,\, F(a)\},
\end{equation}
with an arbitrary initial value $x_0\in (0, H)$.  Then, a.s., for all $n\in \mathbb N$, 
\begin{enumerate}
\item [(i)]  $x_n\le H$;
\item [(ii)]  if in addition $x_0\in (a, H)$ then $x_n\in (a,  H)$.
\end{enumerate}
\end{lemma}

\begin{proof}
If, for some $\omega \in \Omega$ and $n\in \mathbb N$, we have $x_n(\omega)\le H$, then
by  Assumption \ref{as:faH}, (i), and \eqref{cond:mainl},   
\[
x_{n+1}(\omega)=f(x_n(\omega))+l\chi_{n+1}(\omega)\le f^H+l<H.
\]
If, for some $\omega \in \Omega$ and $n\in \mathbb N$, we also have $x_n(\omega) \in (a,H]$, then
by Assumption \ref{as:faH}, (ii), and  \eqref{cond:mainl},
\[
x_{n+1}(\omega)=f(x_n(\omega))+l\chi_{n+1}(\omega)>f(a)-l = f(a)-a+a-l>l+a-l=a.
\]
\end{proof}

\begin{remark}
\label{rem:condlaH} 
Lemma \ref{lem:main0}  implies persistence of solutions with initial values $x_0\in (a, H)$.
\end{remark}

\begin{theorem}
\label{lem:add2}
Let Assumptions \ref{as:chibound} and \ref{as:faH}  hold, $b$ be defined in \eqref{def:b1},
$(x_n)$ be a solution of equation \eqref{eq:main} with $l$ satisfying \eqref{cond:mainl} and
\begin{equation}
\label{cond:lb}
l>b-f(b)=-F(b),
\end{equation}
and $x_0\in (0, H)$.   Then, a.s.,  $x_n$ eventually gets into  the interval $(a, H)$ and stays 
there.
\end{theorem}
\begin{proof}
By Lemma~\ref{lem:main0}, it is sufficient to prove that $x_n\in (a,H)$ for some $n\in {\mathbb N}$, a.s.
Let $\delta>0$ satisfy $l>b-f(b)+\delta$ (in particular, we can take
$\delta=\alpha(l-b+f(b))$ for any $\alpha \in (0,1)$). We define
\begin{equation}
\label{def:probability}
p_1:=\mathbb P\left\{\omega\in \Omega:  \chi(\omega)\in \left(\frac{b-f(b)+\delta}{l}, 1\right)\right\}, \quad K:=\left[ \frac a\delta\right]+1.
\end{equation}
By Lemma \ref{lem:main0} we only have to consider the case  $x_0\in (0, a]$.
Let us note that for any $x_n \in (0, a]$ and $\displaystyle \chi_{n+1} \in \left(\frac{b-f(b)+\delta}{l}, 1\right)$,
we have
\begin{equation*}
\begin{split}
x_{n+1}  =  f(x_n)+l\chi_{n+1} \geq  & f(x_n)-x_n +x_n + l\frac{b-f(b)+\delta}{l} 
\\ \geq & f(b)-b+x_n + b-f(b)+\delta =x_n +\delta.
\end{split}
\end{equation*}
By Assumption~\ref{as:chibound}, $p_1>0$, moreover, the probability
\begin{equation}
\label{def:probability1}
p_K:=\mathbb P\left\{\omega\in \Omega: \chi_i(\omega)\in \left(\frac{b-f(b)+\delta}{l}, 1\right)~~i=1, \dots, K \right\}=p_1^K>0.
\end{equation}
Thus, the probability
\begin{equation*}
\begin{split}
p_{out}:=\mathbb P\left\{\omega\in \Omega: \chi_j(\omega)\in \left[-1,\frac{b-f(b)+\delta}{l} \right] \mbox{~for some~~}j\in \{ 1, \dots, K\} 
\right\} = 1-p_1^K \in (0,1).
\end{split}
\end{equation*}
If all $\chi_i$, $i=j+1, j+2, j+K$, are in $\displaystyle \left(\frac{b-f(b)+\delta}{l}, 1\right)$, then
$$ x_{j+1} \geq x_j+\delta, ~x_{j+2} \geq x_{j+1}+\delta \geq x_j+2\delta, ~~\dots~, x_{j+K}\geq x_j+K\delta > a.$$
By Lemma \ref{lem:main0}, it is sufficient to show that the probability $p_s=0$, where
\begin{equation*}
\begin{split}
\label{def:probability2}
p_{s}:={\mathbb P}\left\{\omega\in \Omega: \mbox{~among any $K$ successive $j$ ~, there is~}\chi_j(\omega)\in \left[-1,\frac{b-f(b)+\delta}{l}
\right]\right\} =0.
\end{split}
\end{equation*}
Let us take some $\varepsilon >0$ and prove that $p_{s}<\varepsilon$. 
Among any $K$ successive $j$,  there  is $\chi_j$ in the above interval with probability $p_{out}<1$.
In particular, there is such $\chi_j$ among $j=1, \dots, K$, with probability $p_{out}$, as well as among $j=K+1, \dots, 2K$,
and in any of non-intersecting sets $j=nK,nK+1, \dots, (n+1)K-1$, $n=0, \dots m-1$. The probability that 
there is $\chi_j$ in the above interval among any $K$ successive $\chi_j$ among $j=1, \dots, mK-1$, is $p_{out}^m$, and $p_s \leq p_{out}^m$. 
Since  $p_{out}^m<\varepsilon$ as soon as $m>\ln \varepsilon/\ln(p_{out})$, we conclude that  $p_{s}=0$,
which completes the proof.
\end{proof}

\begin{corollary}
\label{rem:add1}
Under the assumptions of Theorem~\ref{lem:add2},  if in addition 
we assume $f(x)<x-l$ for $x>H$, then, for any initial condition $x_0\in [0, \infty)$, 
all solutions eventually belong to the interval $(a,H)$.
\end{corollary}

\section{Dynamics Depending on Perturbations (the case $l<b-f(b)$)} 
\label{sec4}

In this section we assume that
\begin{equation}
\label{cond:lless}
l<b-f(b)=-F(b),
\end{equation}
where $b$ is defined in \eqref{def:b1}, and $f$ corresponds to the system with an Allee effect. As we assume an upper bound for the 
perturbation, the dynamics is expected to be dependent on the initial condition: low density if the initial condition is small enough and 
sustainable (persistent) for a large enough initial condition. 
We recall that a solution $(x_n)$ is {\em persistent} if there exist $n_0 \in {\mathbb N}$ and $a>0$  such that 
$x_n>a$ for any $n \geq n_0$. 

In a non-stochastic case, if the system exhibits the Allee effect, then for a small initial condition, the solution tends to zero.
However, in the case of both truncation and stochastic perturbations satisfying Assumption \ref{as:chibound}, the expectation of $x_n$
exceeds a certain positive number. The density function $\phi(x)$ is positive, thus
\begin{equation}
\label{def:alpha}
\alpha:=\int_0^1 x \phi(x)~dx >0.
\end{equation}

\begin{lemma}
\label{lem:add3}
Suppose that Assumption \ref{as:chibound} holds and $f:(0,\infty) \to (0,\infty)$ is a continuous function.
Then the expectation of the solution $(x_n)$ of \eqref{eq:main} is not less than $\alpha$ defined in \eqref{def:alpha}. 
\end{lemma}
\begin{proof}
From \eqref{eq:main}, $x_n \geq \max\{ l \chi_n, 0\}$, thus the expectation of $x_n$ is not less than
$$\int_{-1}^1 l \max\{ x,0\} \phi(x)~dx = \int_{-1}^0 0 \, \phi(x)~dx + \int_0^1 x \phi(x)~dx =\alpha,$$
which concludes the proof.
\end{proof}


\subsection{A.s. persistence and  a.s. low density areas}

Suppose that Assumptions~\ref{as:faH},\ref{as:bmain} hold with $b_1 \geq b$,
where $b$ is denoted in \eqref{def:b1} and $l$ satisfies \eqref{cond:mainl}.   Then we can introduce positive numbers
\begin{equation}
\label{def:ul}
u_l:=\sup\{u<a: F(u)<l\}
\end{equation}
and 
\begin{equation}
\label{def:vl}
v_l:=\inf\{v>b: F(v)> - l\},
\end{equation}
where $F$ is defined in \eqref{def:F}.

\begin{theorem}
\label{lem:vul}
Suppose that Assumptions \ref{as:chibound} - \ref{as:bmain} hold with $b_1 \geq b$,
where $b$ is denoted in \eqref{def:b1} and $l$ satisfies \eqref{cond:mainl}, \eqref{cond:lless}.  
Let $(x_n)$ be a solution to \eqref{eq:main} with an arbitrary initial value $x_0\in [0, H]$.
Let $u_l$ be defined as in \eqref{def:ul} and $v_l$ be defined as in \eqref{def:vl}.  Then the following statements are valid. 
\begin{enumerate}
\item [(i)] $b<v_l<u_l<a$.
\item [(ii)] $F(u_l)=l$, \,\, $ F(v_l)=-l$, \,\, $F(x) \ge l$, for  $x\in (u_l, a)$,  \,\, $F(x)\le -l$, for $x\in (b, v_l)$.
\item [(iii)] If $x_0\in (0, v_l)$, there exists $n_1 \in {\mathbb N}$ such that $x_n \in [0,b]$ a.s. for $n \geq n_1$. 
\item [(iv)] If $x_0\in (u_l, H)$ then  $x$ persists a.s.; moreover, there exists $n_2 \in {\mathbb N}$ such that 
$x_n \in [a,H]$  a.s. for $n \geq n_2$.
\end{enumerate}
\end{theorem}

\begin{proof}
Since
\[
b\in   \{u<a: f(u)-u<l\} , \quad a\in \{v>b:v-f(v)< l\},
\]
both sets in \eqref{def:ul} and \eqref{def:vl} are non-empty and $u_l\le a$, $v_l\ge b$. By continuity of $f$ and 
Assumptions \ref{as:faH},\ref{as:bmain} we have 
\[
u_l< a, \quad v_l> b, \quad F(u_l)=l, \quad F(v_l)=-l.
\]
So $u_l\neq v_l$, 
\[
v_l\in \{u<a: F(u)<l\}  \implies v_l<u_l,
\]
which completes the proof of (i)-(ii).

(iii) Define
\[
\Delta_l(y):=\inf_{x\in [b, y]}\{x-f(x)-l\}.
\]
Note that
\[
\Delta_l(b)=b-f(b)-l>0, \quad \Delta_l(v_l)=0,
\]
and the function   $\Delta_l: [b, v_l]\to [b-f(b)-l, \, 0]$ is non-increasing. 
Then, for each $x_0\in (b, v_l)$ and each $x\in (b, x_0)$, we have
\[
\Delta_l(x_0)\le \Delta_l(x), \quad x-f(x)-l\ge  \Delta_l(x).
\]
So, a.s.,
\[
x_1=f(x_0)+l\chi_1\le f(x_0)+l\le f(x_0)+x_0-f(x_0)-\Delta_l(x_0)=x_0-\Delta_l(x_0).
\]
If $x_1\le b$ we stop. If $x_1> b$, we have, a.s.,
\[
x_2=f(x_1)+l\chi_2\le f(x_1)+l\le f(x_1)+x_1-f(x_1)-\Delta_l(x_1)\le  x_0-\Delta_l(x_0)-\Delta_l(x_1)\le x_0-2\Delta_l(x_0).
\]
Thus, after at most $K$ steps, where
\[
K=\left[ \frac{v_l-b}{\Delta_l(x_0)} \right]+1,
\]
$x_n$ gets into the interval $(0, b)$ and by Lemma~\ref{lem:add1} stays there a.s.

(iv) Define
\[
\tilde \Delta_l(y):=\inf_{x\in [y, a]}\{f(x)-x-l\},
\]
and note that
\[
\tilde \Delta_l(a)=f(a)-a-l>0, \quad \tilde \Delta_l(v_l)=0,
\]
and the function  $\tilde \Delta_l: [v_l, a]\to [0, \, f(a)-a-l]$ is non-decreasing. 
Then, for each $x_0\in (u_l, a)$ and each $x\in (x_0, a)$, we have
\[
\tilde \Delta_l(x_0)\le \tilde \Delta_l(x), \quad f(x)-x-l\ge  \tilde \Delta_l(x).
\]
So, a.s., 
\[
x_1=f(x_0)+l\chi_1\ge f(x_0)-l\ge f(x_0)+\tilde \Delta_l(x_0)-f(x_0)+x_0=x_0+\tilde \Delta_l(x_0).
\]
If $x_1\ge a$ we stop. If $x_1< a$, we have, a.s., 
\[
x_2=f(x_1)+l\chi_2\ge f(x_1)-l\ge f(x_1)+\tilde \Delta_l(x_1)-f(x_1)+x_1\ge x_0+\tilde \Delta_l(x_0)+\tilde \Delta_l(x_1)\ge x_0+2\tilde \Delta_l(x_0).
\]
Thus, after at most $K$ steps, where
\[
K=\left[ \frac{a-u_l}{\tilde \Delta_l(x_0)} \right]+1,
\]
$x_n$ gets into the interval $(a, H)$ and stays there, a.s., by Lemma~\ref{lem:main0}.
\end{proof}


\subsection{Mixed behavior}

So far we have considered the areas starting with which the solution is guaranteed to sustain 
(and be in $[a,H]$) or to stay in the neighbourhood $[0,b]$ of zero. 
Let us consider a more complicated case when a solution can either eventually persist or eventually belong to $[0,b]$.
We single out intervals starting with which a solution can change domains of attraction, switch between persistence and low-density  behavior. In particular, we obtain lower bounds for probabilities that eventually $x_n \in [a,H]$ and
$x_n \in [0,b]$.

As everywhere above, in this subsection we assume that Assumptions \ref{as:chibound}-\ref{as:bmain} 
and conditions \eqref{cond:mainl}, \eqref{cond:lless} hold.  Based on this, we can define
\begin{equation}
\label{def:betaalphal}
\beta_l=\inf\{b<x<a: F(x)>l\}, \quad \alpha_l=\sup\{b<x<a: F(x)<-l\}.
\end{equation}
Note that, since $F(a)>l$, $F(b)<-l$, and $F$ is continuous, both sets in the right-hand-sides 
of formulae in \eqref{def:betaalphal} are non-empty. 

Let $u_l$ and $v_l$ be defined as in \eqref{def:ul} and \eqref{def:vl}, respectively. 
Note  that
\[
v_l<\beta_l\le u_l, \quad v_l\le \alpha_l<u_l,
\]
and 
\[
\max_{x\in [b, \beta_l]}F(x)\le l, \quad \min_{x\in [\alpha_l,a]} F(x) \ge  -l.
\]

The points $a,b,u_l,v_l,\alpha_l,\beta_l$ are illustrated in Figure~\ref{figure_add1}.

\begin{figure}[ht]
\centering
\includegraphics[height=.24\textheight]{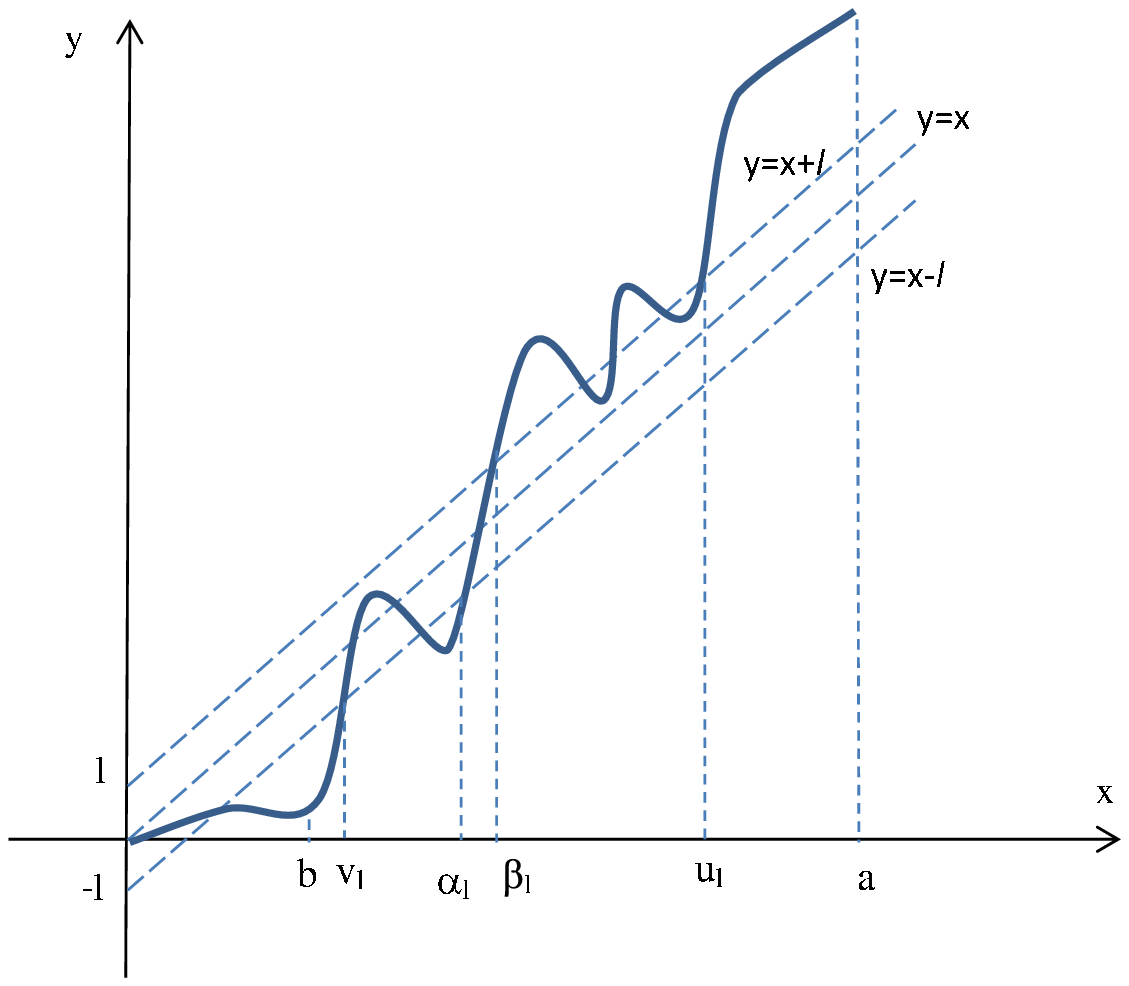}
\hspace{8mm}
\includegraphics[height=.24\textheight]{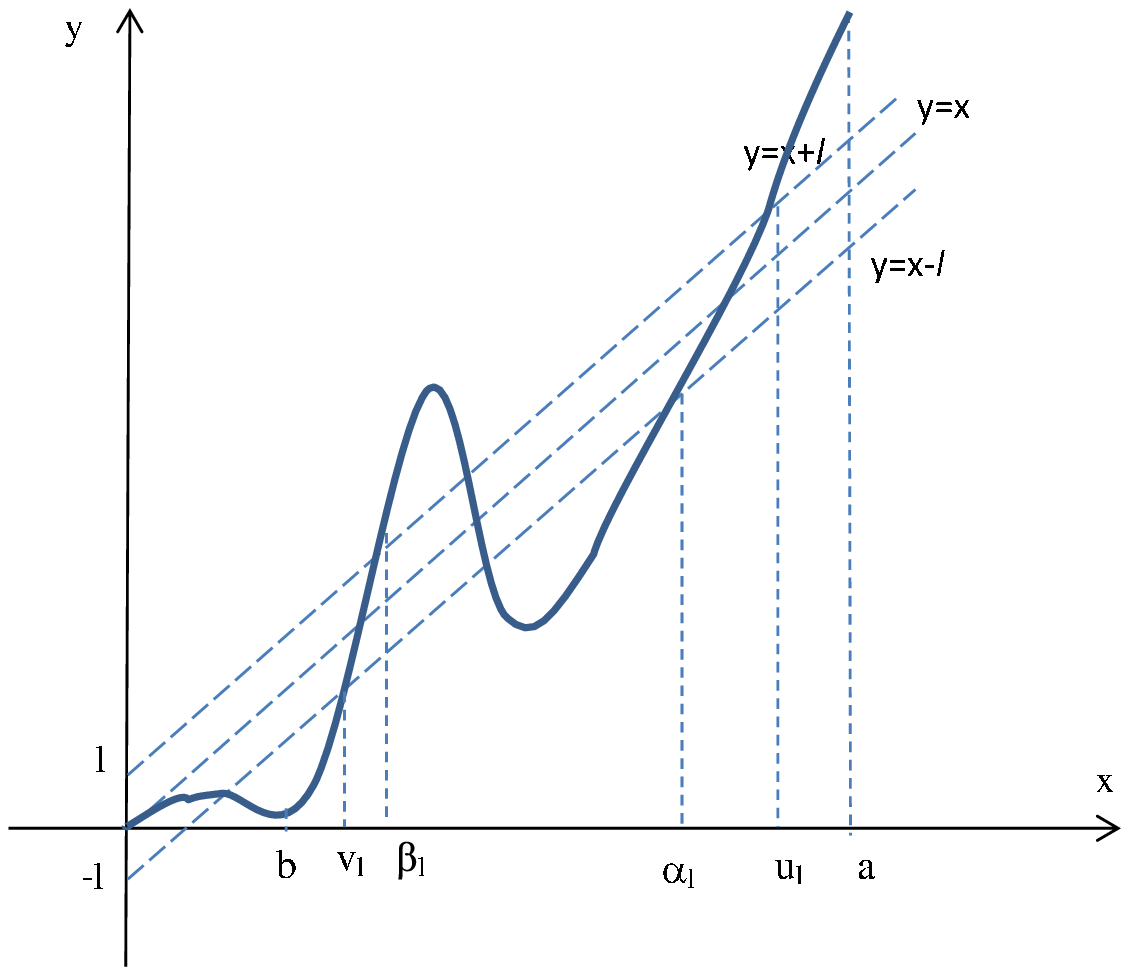}  
\caption{An illustration of the points $a,b,u_l,v_l,\alpha_l,\beta_l$ in the two cases: (left) $\alpha_l < \beta_l$
and (right) $\alpha_l > \beta_l$.
}
\label{figure_add1}
\end{figure}

\begin{remark}
It is possible that $\alpha_l>\beta_l$, see Example~\ref{ex:demo1} and Fig.~\ref{figure_add1}, right. 
However, as $F$ is continuous, $F(b)<-l$, $F(\beta_l)=l$, $F(\alpha_l)=-l$, $F(a)>l$,
the inequality $\alpha_l>\beta_l$ immediately implies that there are at least 3 fixed points of $f$ on $(b,a)$.
In this case we are able to prove only ``essential extinction'' for $x_0\in (v_l, \beta_l)$ and persistence for 
$x_0\in (\alpha_l, u_l)$ (see Lemma~\ref{lemma:mixed1} below).

However, if $\alpha_l<\beta_l$, for each $x_0\in (\alpha, \beta)\subset (\alpha_l, \beta_l)$,  a solution persists with a positive probability and also 
reaches the interval $[0, b]$ with a positive probability.  So solutions with the initial value on the non-empty 
interval $(\alpha_l, \beta_l)$ demonstrate mixed behavior (see Corollary~\ref{corol0} below). 
\end{remark} 

\begin{example}
\label{ex:demo1}
Consider \eqref{eq:main} with 
$$
f(x)= \left\{ \begin{array}{ll}  
\displaystyle \frac{3x}{3+(x-2)^2}, & 0 \leq x \leq 1; \\ \\
x - \sin(\pi(x-1))- \frac{1}{4}, & 1 < x \leq 5; \\ \\
\displaystyle \frac{8.55x}{8+(x-6)^2}, & 5 \leq x. 
\end{array} \right.
$$
We can take $a=5.2$, $f(a) \approx 5.4286$, $H=7$, $f^H<6.8$, $f(H)=6.65>a$, $F(H)=-0.35$.
Here the minimum $-\frac{5}{4}$ of $F(x)$ is first attained at $b=\frac{3}{2}$,
however, $\displaystyle F\left( \frac{7}{2}\right)=- \frac{5}{4}$ as well. We consider $l < \min\{ 
F(a),H-f^H,-F(b)\}$, so we can
take $l < \min\{ 0.2286,0.2,1.25 \}$. Then, it is easy to see that $\alpha_l \in (3.5, 5.2)$, $\beta_l \in \left( \frac{3}{2},\frac{5}{2} 
\right) =(1.5,2.5)$, so $\beta_l < \alpha_l$. There are exactly 4 fixed points of $f(x)=x - \sin(\pi(x-1))- 1/4$ on
$[1,5]$ which are $\arcsin(0.25)/\pi+1$, $2-\arcsin(0.25)/\pi$, $\arcsin(0.25)/\pi+3$, $4-\arcsin(0.25)/\pi$,
and a fixed point $\approx 5.106$ on $(5, 5.2)$.
\end{example}

Let  $x_0 \in (\alpha_l, u_l]$. Define
\begin{equation}
\label{add}
A=A(x_0):=\min_{x\in [x_0,u_l]} F(x) > -l
\end{equation}
and
\begin{equation}
\label{def:prob1}
p_1=p_1(x_0):=\mathbb P\left\{\omega\in \Omega: \chi(\omega)\ge  1-\frac{l+A}{2l} \right\}, \quad 
K_1=K_1(x_0):=\left[\frac{2(u_l-x_0)}{l+A}\right]+1.
\end{equation}

Let  $x_0 \in [v_l, \beta_l)$. Define
\begin{equation}
\label{subtract}
B=B(x_0):=\max_{x\in [v_l,x_0]} F(x) < l
\end{equation}
and
\begin{equation}
\label{def:prob2}
p_2=p_2(x_0):=\mathbb P\left\{\omega\in \Omega: \chi(\omega)\le  -1+\frac{l-B}{2l} \right\}, \quad 
K_2=K_2(x_0):=\left[\frac{2(x_0-v_l)}{l-B}\right]+1.
\end{equation}

\begin{lemma}
\label{lemma:mixed1}
Let Assumptions  \ref{as:chibound}-\ref{as:bmain} hold and $l$ satisfy conditions \eqref{cond:mainl} and \eqref{cond:lless}, 
where $b$ is defined as in \eqref{def:b1}. Let $(x_n)$ be a solution to \eqref{eq:main} with $x_0\in [0, H]$, and
$\alpha_l$,  $\beta_l$ be denoted by \eqref{def:betaalphal}.

Then the following statements are valid.
\begin{enumerate}
\item [(i)] 
If  $x_0 \in (\alpha_l, H]$  then the solution $x_n$ will eventually get  into the 
interval $[a, H]$ with the persistence probability $P_p$ such that  
$$P_p\ge p_1^{K_1},$$
where $p_1$ and $K_1$ are defined in \eqref{def:prob1}.

\item [(ii)] 
If $x_0 \in [0, \beta_l)$  then the
solution $x_n$ will eventually get into the interval $[0, b]$  with the ``low density" (``essential extinction") probability $P_e$ satisfying 
$$ 
P_e \ge p_2^{K_2},
$$
where $p_2$ and $K_2$ are defined in \eqref{def:prob2}.
\end{enumerate}
\end{lemma}

\begin{proof} 
Let $u_l$ and $v_l$ be defined by \eqref{def:ul}, \eqref{def:vl}, respectively.  
By Theorem~\ref{lem:vul}, it is enough to prove (i) for $x_0\in 
(\alpha_l, u_l]$ and  (ii) for $x_0\in [v_l, \beta_l)$.


(i) Let $A$, $K_1$ and $p_1$ be defined, respectively, as in \eqref{add} and \eqref{def:prob1}. We set
\[
\Omega_k:=\left\{\omega\in \Omega: \chi_k(\omega) \geq 1- \frac{l+A}{2l}  \right\} ,
\quad k=1, 2, \dots, K_1, \quad \mathcal A:=\bigcap_{k=1}^{K_1} \Omega_k.
\]
Note that
\[
\mathbb P\left[\mathcal A \right]=p_1^{K_1}.
\]
We prove that  
\begin{equation}
\label{proof:1}
\text{For each $\omega\in \mathcal A$ there exists   a number $n\le K_1$, such that $x_n(\omega)>u_l.$}
\end{equation}
By \eqref{add},  $F(x) \geq A>-l$, for any $x \in [x_0,u_l]$.
Since $\mathcal A\subseteq \Omega_1$ we have, on $\mathcal A$,
\[   
x_1=f(x_0)+l\chi_1=x_0+F(x_0)+l\chi_1 \geq x_0 + A +l\left( 1- \frac{l+A}{2l} \right)= x_0 + \frac{l+A}{2}.
\]
Similarly, for each $k=1, 2, \dots, K_2-1$,  if $x_k\in  [x_0, u_l]$,
since $\mathcal A\subseteq \Omega_k$ and $F(x_k) \geq A>-l$,   we have, on $\mathcal  A$,
\[
x_{k+1} \geq x_k+\frac{l+A}{2}.
\]
The set  $\mathcal A$ can be presented as
\[
\mathcal A=\mathcal A_{11}\cup \mathcal A_{12},\quad \mathcal A_{11}\cap \mathcal A_{12}=\emptyset,
\]
where
\[
\mathcal A_{11}:=\left\{\omega\in \mathcal A: x_1(\omega)>u_l   \right\},
\quad \mathcal A_{12}:=\left\{\omega\in \mathcal A: x_1(\omega)\in \left(x_0 + \frac{l+A}{2}, \,u_l \right]  \right\}.
\]

If $\mathbb P\left[ \mathcal A_{12} \right]=0$, we have, a.s.,  $\mathcal A=\mathcal A_{11}$.
So \eqref{proof:1} holds  a.s. on $\mathcal A$ with $n=1$.
  
If $\mathbb P\left[ \mathcal A_{12} \right]>0$,  we have, on $\mathcal A_{12}$,
  \[
 x_2\geq x_1 + \frac{l+A}{2}\geq x_0+2\frac{l+A}2.
  \]
Presenting  $\mathcal A_{12}$ in the same way as above,
\[
\mathcal A_{12}=\mathcal A_{21}\cup \mathcal A_{22},\quad \mathcal A_{21}\cap \mathcal A_{22}=\emptyset,
\]
where
\[   
\mathcal A_{21}:=\left\{\omega\in \mathcal A_{12}: x_2(\omega)>u_l   \right\},
\quad \mathcal A_{22}:=\left\{\omega\in \mathcal A_{12}: x_2(\omega)\in \left(x_0 + 2\frac{l+A}{2}, \,u_l \right]
\right\},
\]
we consider again two cases: $\mathbb P\left[ \mathcal A_{22} \right]=0$
and $\mathbb P\left[ \mathcal A_{22} \right]>0$.
If $\mathbb P\left[ \mathcal A_{22} \right]=0$, we have, a.s.,
$\mathcal A=\mathcal A_{11}\cup \mathcal A_{21}$, so \eqref{proof:1} holds
with $n=1$ on $\mathcal A_{11}$ and $n=2$ on $\mathcal A_{21}$.
If $\mathbb P\left[ \mathcal A_{22} \right]>0$ we continue the process.
 
Analogously, if  $\mathbb P\left[ \mathcal A_{k-1, 2} \right]>0$, for some $k<K_1$,  we set
\[
\mathcal A_{k-1, 2}=\mathcal A_{k1}\cup \mathcal A_{k2},\quad \mathcal A_{k1}\cap \mathcal A_{k2}=\emptyset,
\]
where
\[   
\mathcal A_{k1}:=\left\{\omega\in \mathcal A_{k-1, 2}: x_k(\omega)>u_l   \right\},
\quad \mathcal A_{k2}:=\left\{\omega\in \mathcal A_{k-1,2}: x_k(\omega)
\in \left(x_0 + k\frac{l+A}{2}, \,u_l \right]  \right\}.
\]
When  $\mathbb P\left[ \mathcal A_{k, 2} \right]=0$, we have, a.s.,
$\mathcal A=\cup_{i=1}^k\mathcal A_{i1}$, so \eqref{proof:1} holds  with $n=i$ on $\mathcal A_{i1}$,
$i=1, 2, \dots, k$.  When  $\mathbb P\left[ \mathcal A_{k, 2} \right]>0$, we continue the process.
However, by \eqref{def:prob1},    $x_0+K_1 \frac{l+A}{2}>u_l$,  so $\mathcal A_{K_1, 2}=\emptyset$. Then
 $\mathcal A$ can be presented as $\cup_{i=1}^k\mathcal A_{i1}$ where $k$ does not exceed  $K_1$.
 This proves \eqref{proof:1}, so the solution reaches the interval $[u_l, H]$ after at most $K_1$ steps with
a probability at least $p_1^{K_1}$.   Application of Theorem~\ref{lem:vul}, (iv), completes the proof of 
(i).
 
Part (ii) can be proved in a similar way. For $B$, $K_2$ and $p_2$ defined, respectively,  
as in \eqref{subtract} and \eqref{def:prob2},  we set
$$
\Gamma_k:=\left\{\omega\in \Omega: \chi_k(\omega)<-1+\frac{l-B}{2l}  \right\}, \quad k=1, 2, \dots, K_2,
\quad \mathcal B:=\bigcap_{k=1}^{K_2} \Gamma_k,$$
and notice that
$$
 \mathbb P\left[\mathcal B \right] =p_2^{K_2}.
$$   
By \eqref{subtract},  $F(x) \leq B<l$, for any $x \in [v_l,x_0]$.
Then,  on $\mathcal  B$, if $x_k\in  [v_l,x_0]$,   $k=1, 2, \dots, K_2-1$, we get
\[
x_{k+1} \leq x_k-\frac{l-B}{2}.
\]
Noting that $x_0-K_2\frac{l-B}{2}<v_l$, we show that for each $\omega\in \mathcal B$, there exists   
a number $n\le K_2$, such that $x_n(\omega)<v_l.$
So the solution reaches the interval $[0, v_l]$ after at most $K_2$ steps with the probability
at least $p_2^{K_2}$.
Application of Theorem~\ref{lem:vul}, (iii), completes the proof of (ii).
\end{proof}

\begin{remark}
Under the assumptions of Lemma \ref{lemma:mixed1},  
\begin{enumerate}
\item [(i)]
the persistence probability $P_p$ and  the ``low density'' probability $P_e$  depend on $x_0$;  
\item [(ii)]
the number $K_1$   indicates the number of steps necessary for a solution $x_n$ with the initial value $x_0\in 
(\alpha_l, u_l)$ to get into 
the interval $(u_l, H]$.  
Respectively, $K_2$ is the number of steps required for a solution with the initial value $x_0\in (v_l, \beta_l)$ to get into the 
interval $(0, v_l)$.
\end{enumerate}
\end{remark}
 
\begin{remark}
Estimations of probabilities $P_p(x_0)$ and $P_e(x_0)$ are far from being sharp. 
They can be improved under the assumption that $F$ is increasing, if, on each step, 
we estimate the new probability to move right $(A+l)/2$ units (respectively, left $(l-B)/2$ units), 
see Theorem \ref{lem:Fmonimpr} and Corollary~\ref{rem:unit} below. 
\end{remark}

\begin{corollary}
\label{corol0}
Let the conditions of Lemma \ref{lemma:mixed1} hold,
and $x_n$ be a solution of \eqref{eq:main} with the initial value $x_0\in [0, H]$.

\begin{enumerate}
\item [(i)] 
If  $\alpha\in (\alpha_l, u_l)$,  we can estimate the persistence probability $P_p(\alpha)$  uniformly   
for all  initial values $x_0 \in [\alpha, H]$.  
Similarly,  if $\beta\in (v_l, \beta_l)$ we can estimate the  ``low density''  probability $P_e(\beta)$  uniformly  for all initial values  $x_0 \in [0, \beta]$.

\item [(ii)]
If $\alpha_l<\beta_l$,  for each $x_0\in (\alpha_l, \beta_l)$ a solution persists with a positive probability and also 
reaches the interval $[0, b]$ with a positive probability.  
For  $(\alpha, \beta)\subset (\alpha_l, \beta_l)$, we can find estimation of $P_p$  
and $P_e$ valid  for all  $x_0 \in (\alpha, \beta)$.
\end{enumerate}
\end{corollary}

\begin{proof}
If $\alpha\in (\alpha_l, a)$ then
$\displaystyle \min_{x\in [\alpha,a]} F(x) > -l$,
and if $\beta\in (b, \beta_l)$ then 
$\displaystyle \max_{x\in [b,\beta]} F(x) < l$.
 
In order to prove (i), we choose 
\begin{equation*}
\begin{split}
&A(\alpha):=\min_{x\in [\alpha,u_l]} F(x) > -l, \quad B(\beta):=\min_{x\in [v_l, \beta]} F(x) <l,\\
&p_1(\alpha):=\mathbb P\left\{\omega\in \Omega: \chi(\omega)\ge  1-\frac{l+A(\alpha)}{2l} \right\}, \quad  \, \, \, \,
K_1(\alpha):=\left[\frac{2(u_l-\alpha)}{l+A(\alpha))}\right]+1,\\
&p_2(\beta):=\mathbb P\left\{\omega\in \Omega: \chi(\omega)\le  -1+\frac{l-B(\beta)}{2l} \right\}, \quad 
K_2(\beta):=\left[\frac{2(\beta-v_l)}{l-B(\beta)}\right]+1.
\end{split}
\end{equation*}
Taking any $x_0\in [\alpha,u_l]$ and following the proof of Lemma \ref{lemma:mixed1}, after at most $K_1(\alpha)$ steps we
have on $\displaystyle \cap_{i=1}^{K_1(\alpha)}\left\{\omega\in \Omega_i: \chi_i(\omega)\ge  1-(l+A(\alpha))/(2l) \right\}$ :
\[
x_n\ge x_0+K_1(\alpha)\frac{l+A(\alpha)}{2} = \alpha+\left(\left[\frac{2(u_l-\alpha)}{l+A(\alpha))}\right]+1  \right)\frac{l+A(\alpha)}{2}\ge \alpha +u_l-\alpha=u_l.
\]
So the persistence probability $P_p$ satisfies the first of the two estimates
\begin{equation}
\label{est:ppe}
P_p\ge p_1(\alpha)^{K_1(\alpha)}, \quad P_e\ge p_2(\beta)^{K_2(\beta)}.
\end{equation}
A similar  estimation can be done for any $x_0\in [v_l, \beta]$, and the "low density"  probability $P_e$ satisfies the 
second estimate in \eqref{est:ppe}. 

Case  (ii) follows from  case (i),  since for any $x_0 \in (\alpha, \beta)$, estimations of both probabilities  
$P_p$  and $P_e$ in 
\eqref{est:ppe} are valid.
\end{proof}

The proof of the following Lemma is straightforward and thus will be omitted.

\begin{lemma}
\label{lem:uvab}
Let the assumptions of Lemma \ref{lemma:mixed1} hold.
\begin{enumerate}
\item  
The inequality
\begin{equation}
\label{cond:Flbound1}
|F(x)|<l, \quad x\in (v_l, u_l),
\end{equation}
is equivalent to  $\beta_l=u_l$ and  $\alpha_l=v_l$. 
\item In particular, condition \eqref{cond:Flbound1} holds  if 
\begin{equation}
\label{cond:f'1}
f(x_2) - f(x_1) > x_2 - x_1 \mbox{\rm ~ for any ~~ } v_l \leq x_1 < x_2 \leq u_l. 
\end{equation}
\end{enumerate}
\end{lemma}

\begin{remark}
\label{rem:monl}
Note that 
\begin{enumerate}
\item [(i)]   $u_l$ is a non-decreasing function of $l$, 
while  $v_l$ is a non-increasing  function of $l$. 
So, for $l_1<l_2<b$ we have $ (v_{l_1}, u_{l_1})\subseteq (v_{l_2}, u_{l_2})$.

\item [(ii)]  $\beta_l$ is a non-decreasing function of $l$, while  $\alpha_l$ is a non-increasing  function of $l$. 

\item [(iii)]  If condition \eqref {cond:Flbound1} holds for some $l=l_1$, it however can fail for some $l=l_2<l_1$ 
(see Example~\ref{ex:demo2}). 

\item [(iv)]  If condition \eqref {cond:f'1} holds for some $l=l_1$  then  \eqref{cond:f'1}, and therefore  \eqref 
{cond:Flbound1},  will be fulfilled for all $l=l_2<l_1$. 
\end{enumerate}
\end{remark}
In the following example we demonstrate the case when \eqref {cond:Flbound1} holds for some $l=l_1$ but does not 
hold  for any smaller $l$.

\begin{example}
\label{ex:demo2}
Consider \eqref{eq:main} with
$$
f(x)= \left\{ \begin{array}{ll}
\displaystyle \frac{16x}{15+(x-3)^2}, & 0\leq x \leq 2, \\ \\
\displaystyle x-\frac{1}{4x} \sin\left(\frac{\pi}{2}x \right), & 1 < x \leq 12, \\ \\
\displaystyle \frac{x-10}{1+(x-13)^2}+11, & 12 < x .
\end{array}\right.
$$
Then $b \approx 0.945$, $F(b) \approx -0.1584$. The maximum of $f(x)$ for $x>12$ is attained at $x \approx 13.162$ and equals $f^H 
\approx 14.081$.
Take $a=12.3$, $H=14.5$, $f(a) \approx 12.5436$, $F(a) \approx 0.2436$, $f(H) \approx 12.3846$, $F(H) \approx -2.1154$, $H-f^H \approx 
0.419$,  
then we can take any $l<0.24$. On $[2,12]$ local maxima of $F(x)=1/12,1/28,1/44$ are attained at $x=3,7,11$, respectively,
and local minima of $F(x)=-1/20,-1/36$ at $x=5,9$, respectively. Thus for $l \in (1/12,6/25)$, inequality 
\eqref{cond:Flbound1} holds while for $l \in (0, 1/12)$ it fails.
\end{example}

Theorem~\ref{lem:vul}, Lemma~\ref{lemma:mixed1} and Corollary \ref{corol0} imply the following result.

\begin{theorem}
\label{thm:mixed}
Suppose that Assumptions \ref{as:chibound} - \ref{as:bmain} hold, $b$ be  denoted in \eqref{def:b1}, $l$ satisfies conditions 
\eqref{cond:mainl} and \eqref{cond:lless}, and condition \eqref{cond:Flbound1} holds.  
Let $u_l$ be defined as in \eqref{def:ul} and $v_l$ be defined as in \eqref{def:vl}, 
$(x_n)$ be a solution of \eqref{eq:main} with $x_0\in [0, H]$.

Then the following statements are valid.
\begin{enumerate}
\item [(i)] If $x_0\in (0, v_l)$ then there exists $n_1 \in {\mathbb N}$ such that $x_n \in [0,b]$ a.s. for $n \geq n_1$.
\item [(ii)] If $x_0\in (u_l, H)$ then $x$ persists a.s.; moreover, there exists $n_2 \in {\mathbb N}$ such that $x_n \geq a$ 
a.s. for $n \geq n_2$.
\item [(iii)] If $x_0\in (v_l, u_l)$ then $x$ persists with a positive probability 
and eventually belongs to $(0, b)$ with  a positive probability.
\end{enumerate}
\end{theorem}

\begin{lemma}
\label{lem:contuvl}
Let assumptions of Theorem~\ref{lem:vul} hold, and the density  $\phi$ be bounded on $[-1, 1]$ by some $C>0$:
\[
\phi(x) \le C, \quad x\in [-1, 1].
\]
Then $P_e(x_0)\to 1$ as $x_0\downarrow v_l$ and $P_p(x_0)\to 1$ as $x_0\uparrow u_l$.
\end{lemma}
\begin{proof}
Let us prove that $P_p(x_0)\to 1$ as $x_0\uparrow u_l$. 
The other case can be treated similarly. 

By uniform continuity of $F$ on the interval $[0, H]$, for any $\varepsilon \in (0, 2C)$  we can find $\delta_1=\delta_1(\varepsilon)$ such that
\[
|F(x) - F(y)| \le \frac{l\varepsilon}{2C} \quad \text{for} \quad |x-y|<\delta_1, \, \, \forall x, y\in [0, H].
\]
Let
\[
\delta=\delta(\varepsilon)\le \min\left\{\delta_1(\varepsilon), \, \frac{l\varepsilon}{2C}\right\}
\]
and 
\[
\Omega^{(1)}_\varepsilon:=\left\{\omega\in \Omega: \chi_1(\omega)\ge  -1+\frac{\varepsilon}{C} \right\}.
\]
Note that since $\frac{\varepsilon}{C} <2$, the set $\Omega^{(1)}_\varepsilon$ is non-empty and 
\[
\mathbb P\left\{\Omega^{(1)}_\varepsilon \right\}=\int ^{1}_{-1+\varepsilon/C}\phi(s)ds=1-\int_{-1} ^{-1+\varepsilon/C}\phi(s)ds\ge 
1-\varepsilon.
\]
Let 
$0<u_l-x_0<\delta$, then 
\[
|l-F(x_0)|=|F(u_l)-F(x_0)|\le  \frac{l\varepsilon}{2C}, \quad \text{or} \quad F(x_0)\ge l- \frac{l\varepsilon}{2C},
\]
and, on $\Omega^{(1)}_\varepsilon$,  we have $x_1 \in (u_l,H)$, since
\[
x_1=x_0+F(x_0)+l\chi_1\ge x_0+l-  \frac{l\varepsilon}{2C}+l\left(-1+\frac{\varepsilon}{C}\right)
=x_0+ \frac{l\varepsilon}{2C}>u_l-\delta+ \frac{l\varepsilon}{2C}_l \geq u_l.
\]
This  implies 
\[
P_p(x_0)\ge \mathbb P\left\{\Omega^{(1)}_\varepsilon \right\}\ge 1-\varepsilon,\quad \text{whenever}\quad 
0<u_l-x_0<\delta, \]
which completes the proof.
\end{proof}


\subsection{ $F$ is increasing on $(b,a)$}

When $F$ is increasing on $(b,a)$, we can state the following corollary of Theorem \ref{thm:mixed}
and Lemma~\ref{lemma:mixed1}, since in \eqref{add} and \eqref{subtract} we have $A=F(x_0)=B$.

\begin{corollary}
\label{cor:Fmonot}
Let, in addition to assumptions of Theorem \ref{thm:mixed}, the  function $F$ be increasing on $[b, a]$.   Then, 
for each $l\in (0, -F(b))$, we have
\begin{enumerate}
\item [(i)] $u_l=F^{-1}(l)$, $v_l=F^{-1}(-l)$.
\item [(ii)] If $x_0\in (0, v_l)$ then there exists $n_1 \in {\mathbb N}$ such that $x_n \in [0,b]$ a.s. for $n \geq n_1$.
\item [(iii)] If $x_0\in (u_l, H)$ then $x$ persists a.s.; moreover, there exists $n_2 \in {\mathbb N}$ such that $x_n \geq a$
a.s. for $n \geq n_2$.
\item [(iv)] If $x_0\in (v_l, \, u_l)$ then $x$ persists with a positive probability $P_p(x_0)\ge p_1^{K_1}$ and eventually belongs to $(0, b)$ with a positive probability $P_e(x_0)\ge p_2^{K_2}$, where
\[
p_1=p_1(x_0)=\mathbb P \left\{\omega\in \Omega: \chi(\omega)>1-\frac{l+F(x_0)}{2l}  \right\}, \quad 
K_1=K_1(x_0):=\left[\frac{2(u_l -x_0)}{l+F(x_0)}\right]+1;
\]
and 
\[
p_2=p_2(x_0)=\mathbb P \left\{\omega\in \Omega: \chi(\omega)<-1+\frac{l-F(x_0)}{2l}  \right\}, \quad 
K_2=K_2(x_0):=\left[\frac{2( x_0-v_l)}{l-F(x_0)}\right]+1.
\]
\end{enumerate}
\end{corollary}

In the following Theorem we improve estimations of persistence and low-density behavior probabilities $P_p(x_0)$ and $P_e(x_0)$, when $x_0\in (v_l, \, u_l)$.
The estimates are based on evaluating at each step  
the new probability to move right $(F(x_0)+l)/2$ units (respectively, left $(l-F(x_0))/2$ units).
Let us introduce the following notation:
\begin{equation}
\label{def:epsmain}
\varepsilon := \frac{l+F(x_0)}{2}, ~~ K_1 := \left[ \frac{u_l-x_0}{\varepsilon} \right] +1,~~
\delta:= \frac{l-F(x_0)}{2}, \quad K_2 := \left[ \frac{x_0-v_l}{\delta} \right] +1,
\end{equation}
\begin{equation}
\label{def:epsils}
\varepsilon_0:= (F(x_0)+l)/(2l) \in (0,1), \quad \varepsilon_{i} := \frac{l+2F(x_0+ (i-1) \varepsilon)-F(x_0)}{2l},
\quad i=1, \dots, K_1,
\end{equation}
\begin{equation}
\label{def:mus}
\delta_0:= (l-F(x_0))/(2l) \in (0,1), \quad \delta_{i} := \frac{l-2F(x_0+ i \varepsilon)+ F(x_0)}{2l},
\quad i=1, \dots, K_2,
\end{equation}
\begin{equation}
\label{def:probab_1}
\lambda_i :=\mathbb P\{\omega \in \Omega: \chi(\omega) > 1-\varepsilon_i \}=
\int_{\max\{-1,1-\varepsilon_i\}}^1 \phi(t)dt,
\end{equation}
\begin{equation}
\label{def:probab_2}
\mu_i :=\mathbb P\{\omega \in \Omega: \chi(\omega) < -1 + \delta_i \}=
\int_{-1}^{\min\{-1+\delta_i,1\}} \phi(t)dt.
\end{equation}

\begin{theorem}
\label{lem:Fmonimpr}
Assume that Assumptions \ref{as:chibound} - \ref{as:bmain} hold,  $b$, $u_l$  and $v_l$ are  denoted in \eqref{def:b1},
\eqref{def:ul} and \eqref{def:vl}, respectively, and $l$ satisfies 
conditions \eqref{cond:mainl} and \eqref{cond:lless}.  
If the  function $F$ increases on $[v_l, u_l]$ then
a solution to \eqref{eq:main} with the initial value $x_0\in [0, H]$ persists with a positive probability 
\begin{equation}
\label{def:K1}
P_p(x_0)\ge \prod_{i=1}^{K_1} \lambda_i
\end{equation}
and eventually belongs to $(0, b)$ with a positive probability 
\begin{equation}
\label{def:K2}   
P_e(x_0)\ge \prod_{i=1}^{K_1} \mu_i,
\end{equation}   
where $K_1$ and $K_2$ are introduced in \eqref{def:epsmain}, while $\lambda_i$ and $\mu_i$ are denoted in 
\eqref{def:probab_1} and \eqref{def:probab_2}, respectively.
\end{theorem}
  
\begin{proof}
Denote
$
\Omega_i:=\{\omega\in \Omega: \chi_i(\omega) > 1-\varepsilon_i \}, 
$
then
${\mathbb P}\{ \Omega_i \} =\lambda_i$. 
On $\Omega_1$, we have
\begin{equation}
\label{est:omega1}
x_1=x_0+F(x_0)+l\chi_1 \geq x_0+F(x_0)+l-\frac{l+F(x_0)}{2}=x_0+\frac{l+F(x_0)}{2}, \quad \text{or} \quad x_1-x_0 \geq \varepsilon.
\end{equation}
Further, assume that on $\displaystyle \cap_{j=1}^i\Omega_j $ we have $x_i \geq x_0+ i \varepsilon$. 
Then on $\displaystyle \cap_{j=1}^{i}\Omega_j $,  either  
$x_i \ge  u_l$ or  $x_i < u_l$. In the former case, by Theorem~\ref{lem:vul},  $x$ persists and 
\[
{\mathbb P} \left\{ x_{K_1} \geq a \right\}={\mathbb P} \left\{ x_{i} \geq u_l \right\} \geq {\mathbb P} \left\{ \cap_{j=1}^{i}\Omega_j    
\right\} = \prod_{i=1}^{i} \lambda_j \ge \prod_{i=1}^{K_1} \lambda_j.\]
In the latter case,  due to monotonicity of $F$, 
we have
\begin{equation*}
\begin{split}
x_{i+1}=x_i+F(x_i)+l\chi_{i+1} > & x_i+F(x_0+i \varepsilon)+ l- \frac{l+2F(x_0+ i \varepsilon)-F(x_0)}{2}
\\ = &x_i+\frac{l+F(x_0)}{2}=x_i+\varepsilon> x_0+(i+1) \varepsilon.
\end{split}
\end{equation*}
By induction, either $x_i\ge u_l$ for some $i=1, \dots, K_1$ or $x_i \geq x_0+ i \varepsilon$, for all $i=1, \dots, K_1$, and hence on 
$\cap_{j=1}^{K_1}\Omega_j $,
$x_{K_1} \geq u_l.$

To conclude the estimate for $P_p(x_0)$, by Theorem~\ref{lem:vul}, part (iv), for a given $x_0$, we have 
$$
{\mathbb P} \left\{ x_{K_1} \geq a \right\} ={\mathbb P} \left\{ x_{K_1} \geq u_l \right\} \geq {\mathbb P} \left\{ \cap_{j=1}^{K_1}\Omega_j  
\right\}
= \prod_{i=1}^{K_1} \lambda_i.
$$
The estimate for $P_e$ is justified similarly. 
\end{proof}

Both estimates for probabilities  $P_p(x_0)$ and  $P_e(x_0)$ in Corollary \ref{cor:Fmonot}  
can be writen in a more explicit form in the case when the density $\phi$ is bounded below by the constant $h>0$, 
function $F$ is   differentiable on $[b, a]$ and its derivative is bounded from below.

\begin{corollary}
\label{rem:unit}
\begin{enumerate}
\item [(i)] If for some $h>0$ and all $x\in [-1,1]$
\begin{equation}
\label{cond:hphi}
\phi(x)\ge h,
\end{equation}
then the estimates \eqref{def:K1} and  \eqref{def:K2} lead to the inequalities
\[
P_p(x_0)\ge h^{K_1}\prod_{i=1}^{K_1} \varepsilon_i, \quad P_e(x_0)\ge h^{K_2}\prod_{i=1}^{K_2} \delta_i.
\]

\item [(ii) ]
Let, in addition to \eqref{cond:hphi}, for some $\kappa>0$ and all $x, y\in [v_l, u_l]$,
\begin{equation}
\label{cond:kF}
|F(y)-F(x)|\ge \kappa |x-y|.
\end{equation}
Then  estimates \eqref{def:K1} and  \eqref{def:K2} imply
\[
P_p(x_0)\ge h^{K_1}\prod_{i=1}^{K_1} \left(\varepsilon_0+\frac {\kappa (i-1) \varepsilon}l\right), \quad P_e(x_0)\ge h^{K_2}\prod_{i=1}^{K_2} \left(\delta_0+\frac {\kappa (i-1) \delta}l\right),
\]
and substitution of values from \eqref{def:epsmain}-\eqref{def:mus} implies
\[
P_p(x_0)\ge h^{K_1}\left(\frac{\varepsilon}l\right)^{K_1}\prod_{i=1}^{K_1} \left(1+\kappa (i-1)\right), \quad 
P_e(x_0)\ge h^{K_2}\left(\frac{\delta}l\right)^{K_2}\prod_{i=1}^{K_2} \left(1+\kappa (i-1)\right).
\]
\item[(iii)] If, in addition to conditions of (ii),  $\chi$ are uniformly distributed, then $h=1/2$ and 
estimates \eqref{def:K1} and  \eqref{def:K2} take forms
\[
P_p(x_0)\ge \left(\frac{\varepsilon}{2l}\right)^{K_1}\prod_{i=1}^{K_1} \left(1+\kappa (i-1)\right), 
\quad P_e(x_0)\ge \left(\frac{\delta}{2l}\right)^{K_2}\prod_{i=1}^{K_2} \left(1+\kappa (i-1)\right).
\]
\end{enumerate}
\end{corollary}
\begin{proof}
We only have to prove the estimates of $\varepsilon_i$ in (ii)
\begin{equation*}
\begin{split} \varepsilon_i= &\frac{l+2F(x_0+ (i-1) \varepsilon)-F(x_0)}{2l}=\\&\frac{l+F(x_0)}{2l}+\frac{F(x_0+ (i-1) 
\varepsilon)-F(x_0)}{l}\ge\varepsilon_0+\frac {\kappa (i-1) \varepsilon}l,
\end{split}
\end{equation*}
and note that $\varepsilon_0=\frac{\varepsilon}{l}$. The estimates are valid since 
$x_0+(i-1)\varepsilon\le u_l$.
\end{proof}


\section{Multistability}
\label{sec:multi}

So far we have considered only one bounded open subinterval $(a,H) \subset (0,\infty)$, which $f$ mapped into 
$(a+l,H-l)$. However, there may be several non-intersecting subintervals with this property.

\begin{assumption}
\label{as:faHmulti}
Assume that $f:[0,\infty)\to [0, \infty)$ is continuous, $f(0)=0$, $f(x)>0$ for $x>0$ and 
there exist positive numbers $a_i$ and  $H_i$, $a_i<H_i$, $i=1, \dots, k$ and $H_i<a_{i+1}$, 
$i=1, \dots, k-1$, such that
\begin{enumerate}
\item [(i)] $ f_i := \max_{x\in (0,H_i)}f(x) < H_i$, $i=1, \dots,k$;
\item[(ii)]  $f(x)>f(a_i)>a_i, \quad x\in (a_i, H_i]$, $i=1, \dots,k$.
\end{enumerate}
\end{assumption}

\begin{lemma}
\label{lem:add2multi}
Let Assumptions \ref{as:chibound} and \ref{as:faHmulti}  hold,
$(x_n)$ be a solution of equation \eqref{eq:main} with $l$ satisfying, 
for some particular $i \in \{1, 2, \dots, k\}$, 
\begin{equation}
\label{cond:mainlmulti}
l < \min\{H_i-f_i, \,\, f(a_i)-a_i\}.
\end{equation}
If $x_0 \in (a_i,H_i)$ then $x_n\in (a_i,H_i)$.

If in addition 
\begin{equation}
\label{cond:lbmulti}
l>\max_{x \in [0,a_1]} (-F(x))
\end{equation}
then, for an arbitrary  initial value $x_0\in (0, H_1)$, 
a.s.,  $x_n$ eventually gets into  the interval $(a_1, H_1)$ and stays there.
\end{lemma}

\begin{proof}
Let $x_0 \in (a_i,H_i)$, then by \eqref{cond:mainlmulti}
$$
x_1 = f(x_0) + l \chi_1(\omega) <H_i-l+l=H_i
$$
and
$$
x_1 = f(x_0) + l \chi_1(\omega) > a_i+l-l=a_i.$$
Similarly, $x_n \in (a_i,H_i)$ implies $x_{n+1} \in (a_i,H_i)$, the induction step concludes the proof of the first part.

If,  in addition,  \eqref{cond:lbmulti} holds and $x_0 \in [0,a_1)$ then the result follows from Theorem~\ref{lem:add2},
where we assume $a_1=a$, $H_1=H$. Then all the conditions of Lemma~\ref{lem:add2} are satisfied, and, a.s., 
$x_n$ eventually gets into  the interval $(a_1, H_1)$ and stays there, which completes the proof.
\end{proof}

\begin{example}
\label{ex:demo3}
Consider \eqref{eq:main} with $f(x)=x - \sin  x$.
There is the Allee effect on $[0,\pi]$. The function $f(x)$ is monotone increasing, satisfies $f(x)<x$ on
$(2\pi k, (2k+1)\pi k)$, $k =0,1,\dots$ and $f(x)>x$ for $x \in ((2k-1)\pi, 2\pi k)$, $k \in {\mathbb N}$.
Each of the intervals $(\pi k,\pi(k+1))$ is mapped onto itself. For example, we can choose 
$$
a_k \in \left((2k-1)\pi,\left(2k-\frac{3}{4}\right)\pi\right),\quad H_k \in \left(\left(2k+\frac{3}{4}\right)\pi,(2k+1)\pi\right), \quad k \in {\mathbb N}.
$$
By Lemma~\ref{lem:add2multi}, for appropriate $l$, once $x_0\in (a_k,H_k)$, we have $x_0\in (a_k,H_k)$, $k \in {\mathbb N}$.

If $l=0$ (the deterministic case) and $x_0 \in ((2 k-1)\pi,(2k+1)\pi)$ then
$x_n \to 2\pi k$ as $n \to \infty$. 
\end{example}

\begin{example}
\label{ex:demo4}
Consider \eqref{eq:main} with the function $f(x)=x-\sin x +0.5x \sin x$, which experiences the Allee effect and 
multistability.
However, $F(x)=(0.5 x-1)\sin x$ is unbounded, and it is hardly possible to find disjoint intervals $(a_i,H_i)$ mapped into themselves
such that
$$ \min_{x \in [a_i,H_i]} f(x) > H_{i-1}, \quad \max_{x \in [a_i,H_i]} f(x) < a_{i+1},\quad i \in {\mathbb N}.
$$
\end{example}

\section{Numerical Examples}
\label{sec6}

The equations in Examples~\ref{ex:pers1} and \ref{ex:pers2} satisfy Assumptions \ref{as:chibound}, \ref{as:faH} and
\ref{as:bmain}.

As model examples, we can consider \eqref{eq:boukal1} and \eqref{eq:boukal2}.

\begin{example}
\label{ex:pers1}
Consider \eqref{eq:main} with
\begin{equation}
\label{def:4f}
f(x):=\frac{4x}{2+(x-3)^2}, \quad x>0.
\end{equation}
The fixed points of $f$ in \eqref{def:4f} are $c=3-\sqrt{2} \approx 1.586$ and $d=3+\sqrt{2} \approx 4.414$.
The maximum $f_m \approx 6.317$ is attained at $x_m=\sqrt{11} \approx 3.317$. Also, $f(f_m) \approx 1.943$
and the value of $$d_1= \{ x> d: f(x)=c \} \mbox{   ~is ~ } d_1=\frac{11}{3-\sqrt{2}} \approx 6.937.$$
Let us choose $a=1.8$,  $H=6.5$, then $f(a)\approx 2.093$, $f(H) \approx 1.825$, then $F(a)= 0.293$,
$F(H) \approx -4,675$. We consider $l=0.2<0.293$, for illustration of \eqref{def:4f} see Fig.~\ref{figure1}.

Furthermore, $b \approx 0.907$, and $F(b) \approx -0.3384$. 
For any $l<0.293$, there is a domain $(0,v_l)$, starting with which we have low density behavior,
and $(u_l,H)$ which eventually leads a.s. to $(a,H)$.
Let us take $l=0.2$, then $u_l \approx  1.74$, $v_l \approx 0.361$.


\begin{figure}[ht]
\centering
\includegraphics[height=.25\textheight]{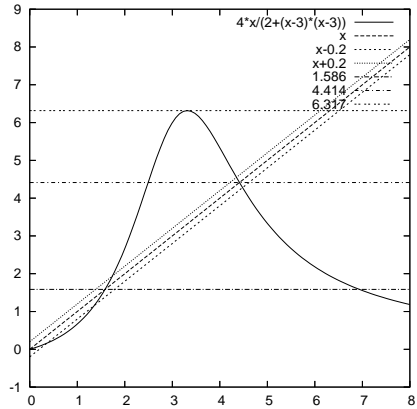}
\caption{The graph of the function in (\protect{\ref{def:4f}});
the fixed points are $c \approx 1.586$ and $d\approx 4.414$, 
the maximum $\approx 6.317$ is attained at $ \approx 3.317$.}
\label{figure1}
\end{figure}


For \eqref{eq:main} with $f$ as in \eqref{def:4f},  $l=0.2$ and $x_0 \in [0,v_l)=[0,0.36)$ we have low density behavior 
(Fig.~\ref{figure2}, left), for $x_0 \in (u_l,H]=(1.74,6.5]$ we have persistence (Fig.~\ref{figure2}, right). If $x_0 \in 
(v_l,u_l) \approx (0.361,1.74)$, then solutions can either sustain or have eventually low density 
(Fig.~\ref{figure2}, middle). 
All numerical runs correspond to the case when $\chi$ has a uniform distribution on $[-1,1]$.

\begin{figure}[ht]
\centering     
\includegraphics[height=.15\textheight]{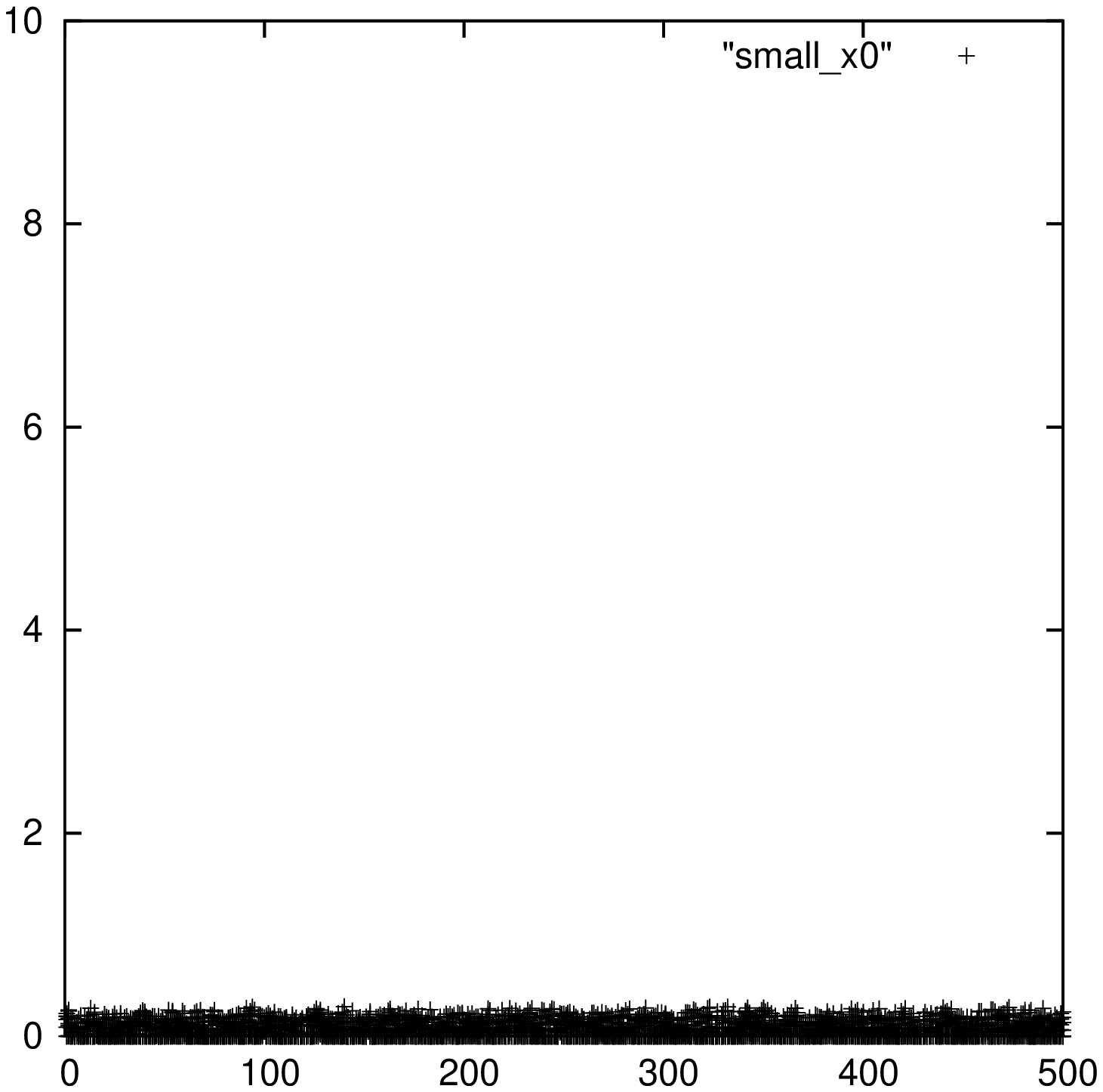}  
\hspace{5mm}
\includegraphics[height=.15\textheight]{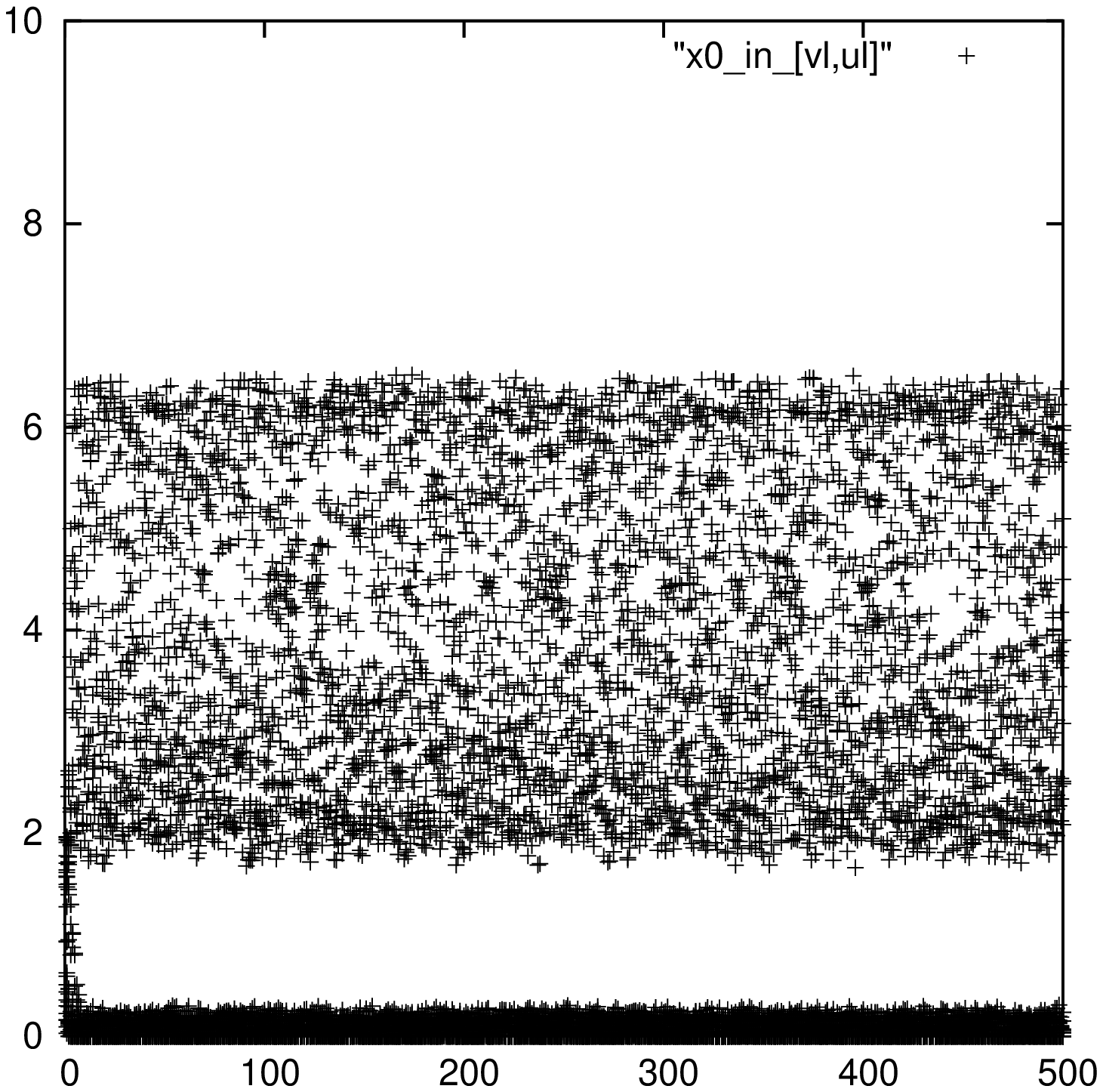}
\hspace{5mm}
\includegraphics[height=.15\textheight]{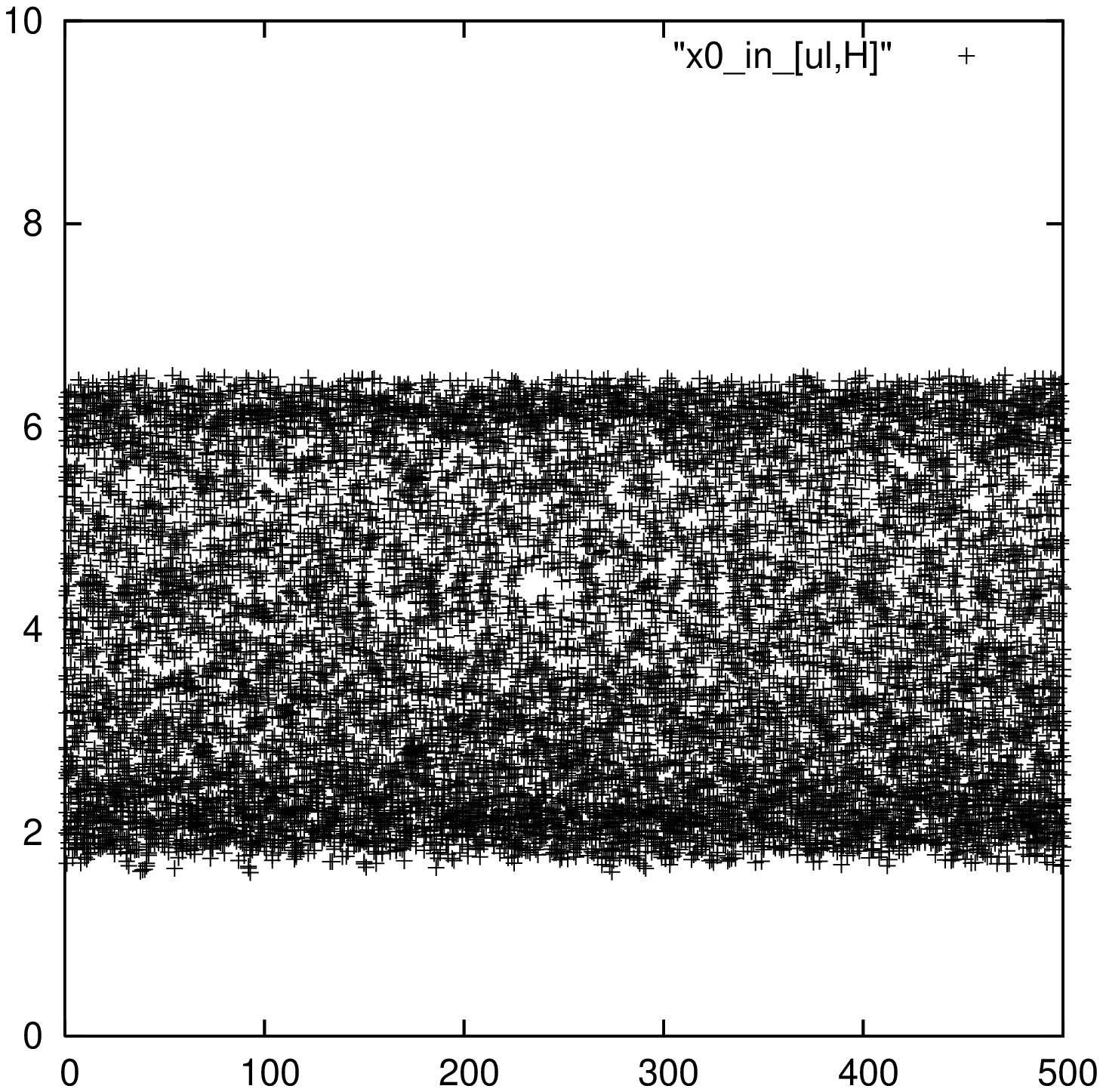}
\caption{Several runs of (\protect{\ref{eq:main}}) with $f$ as in (\protect{\ref{def:4f}}) 
for $x_0 \in [0,0.36)$ (left), $x_0\in (0.361,1.74)$ (middle) and $x_0\in 
(1.74,6.5]$ (right) for $l=0.2$.}
\label{figure2}
\end{figure}

Let us illustrate the dependency of the probability of the solution to sustain on the initial point $x_0 \in
(u_l,v_l)$. Fig.~\ref{figure3} presents 10 random runs starting with $x_0=1.4,1.5,1.6,1.7$ (Fig.~\ref{figure3}, from left
to right).

\begin{figure}[ht]
\centering
\includegraphics[height=.12\textheight]{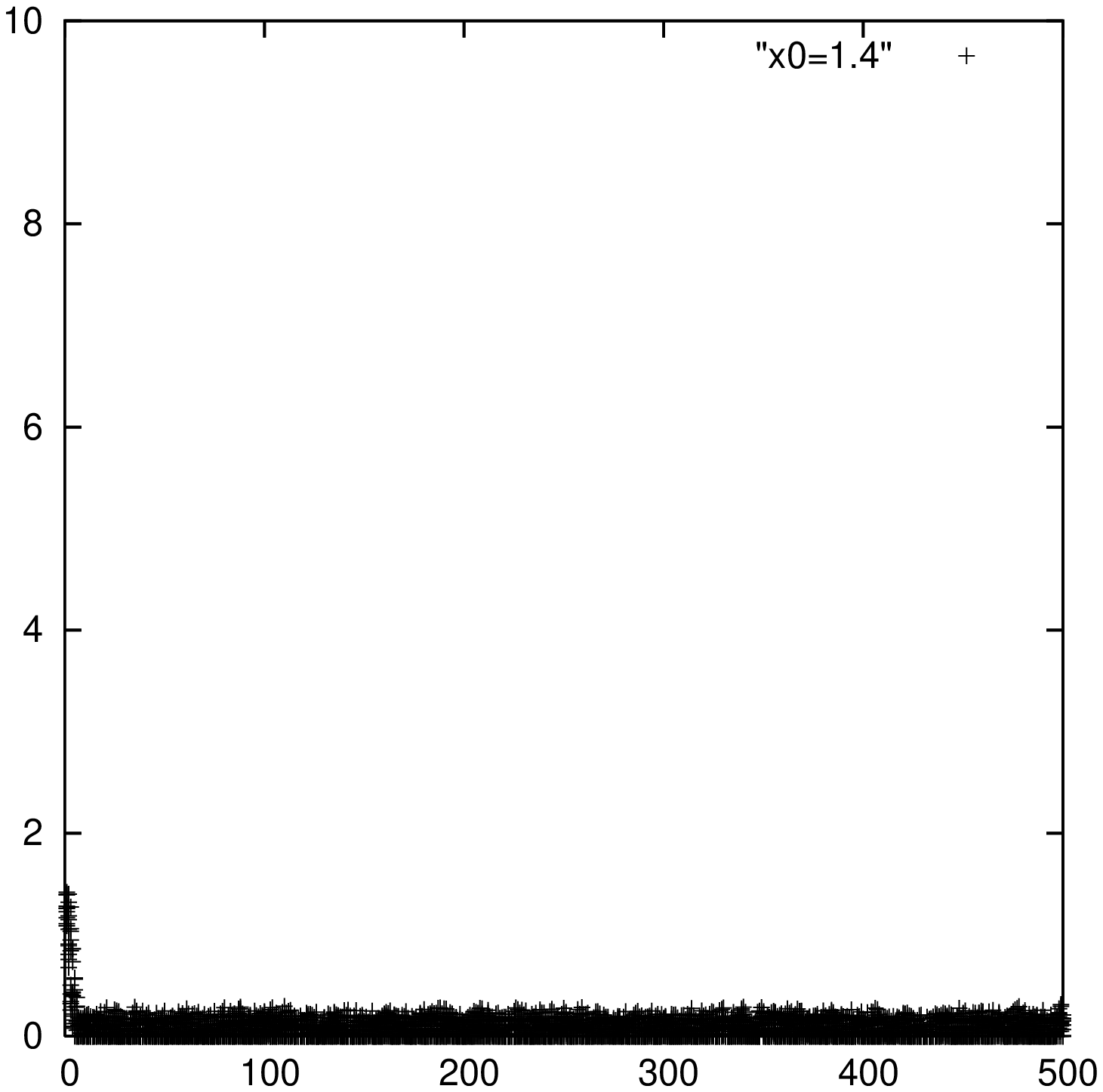}
\hspace{3mm}
\includegraphics[height=.12\textheight]{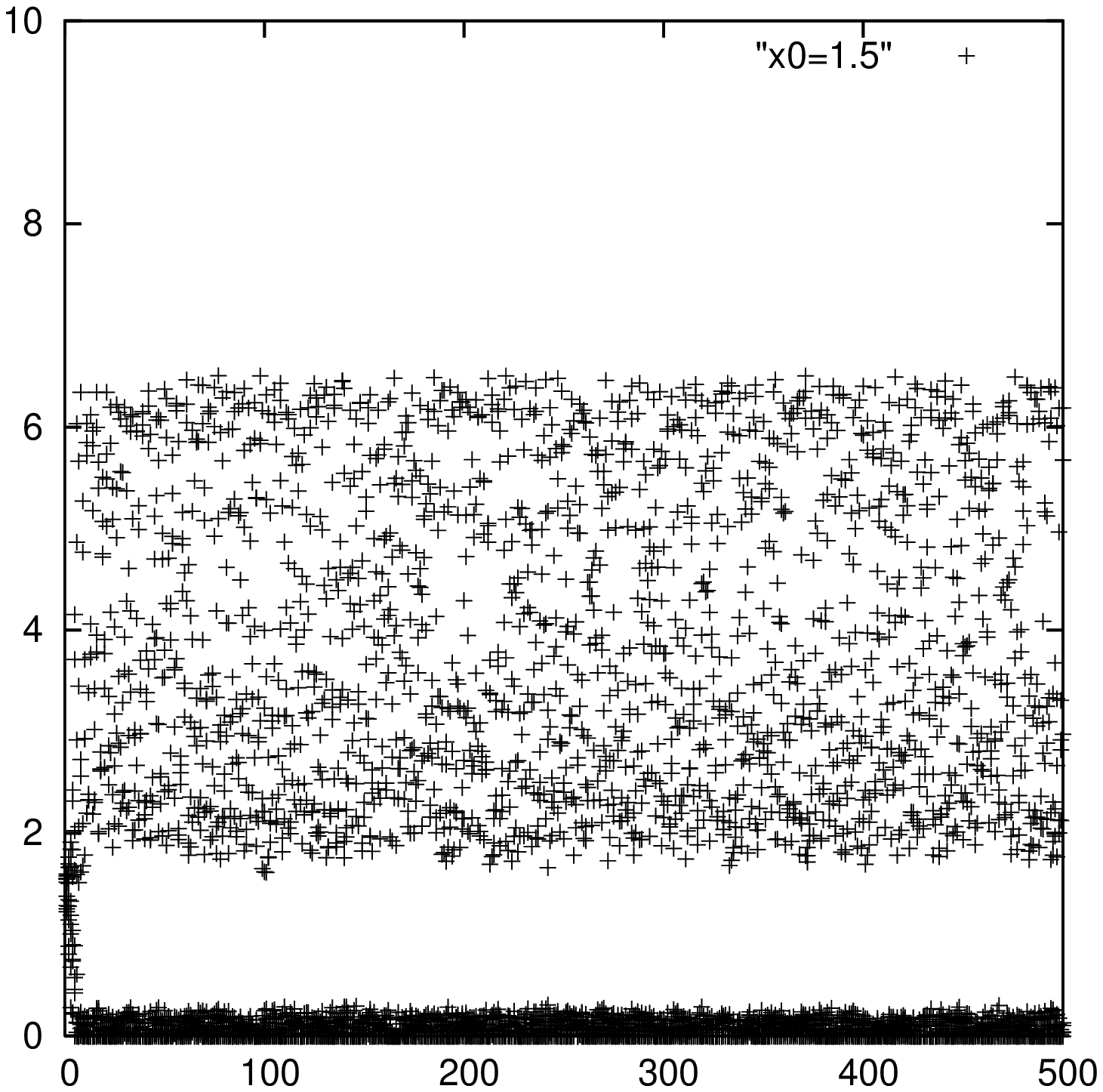}
\hspace{3mm}
\includegraphics[height=.12\textheight]{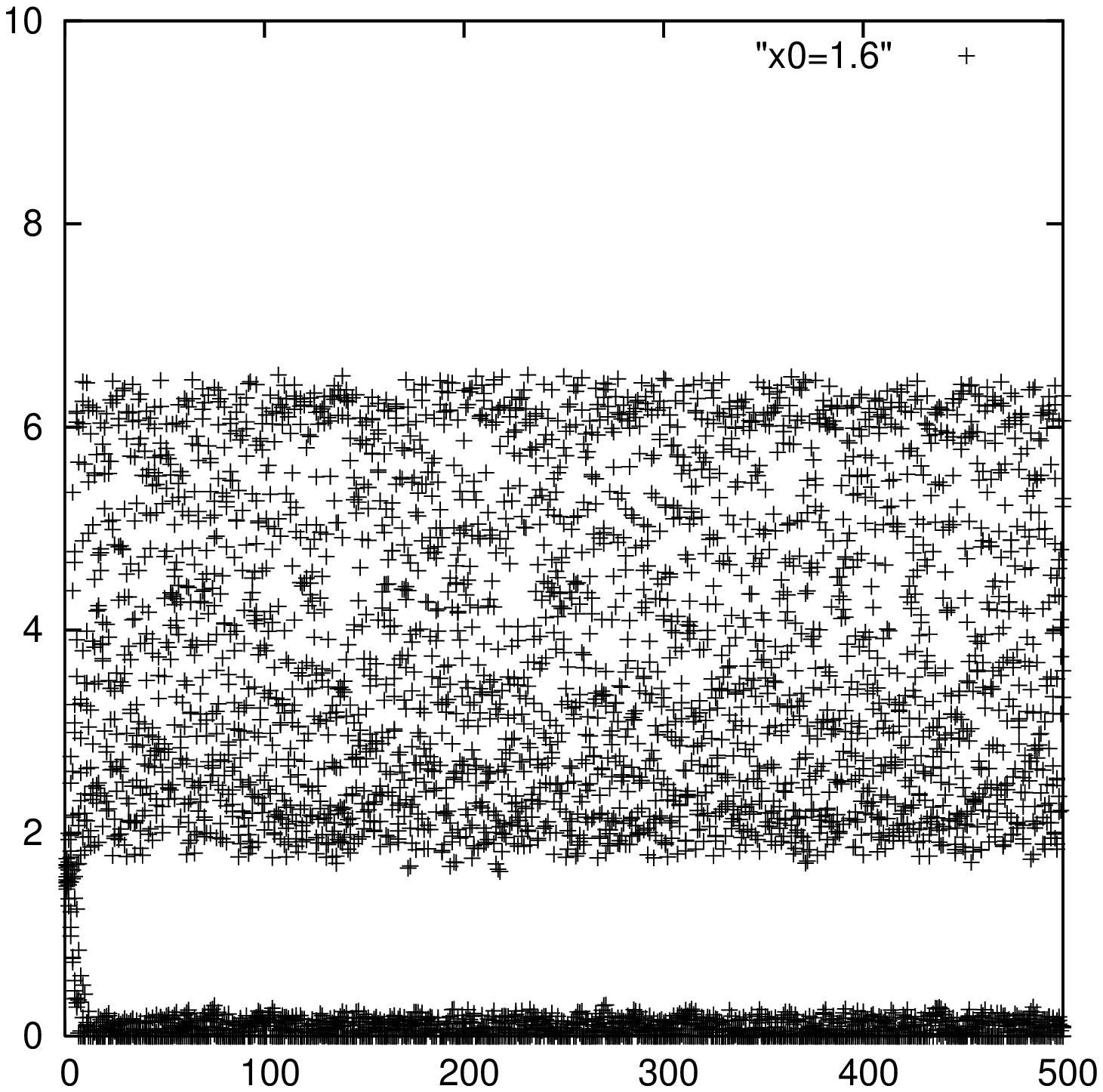}
\hspace{3mm}
\includegraphics[height=.12\textheight]{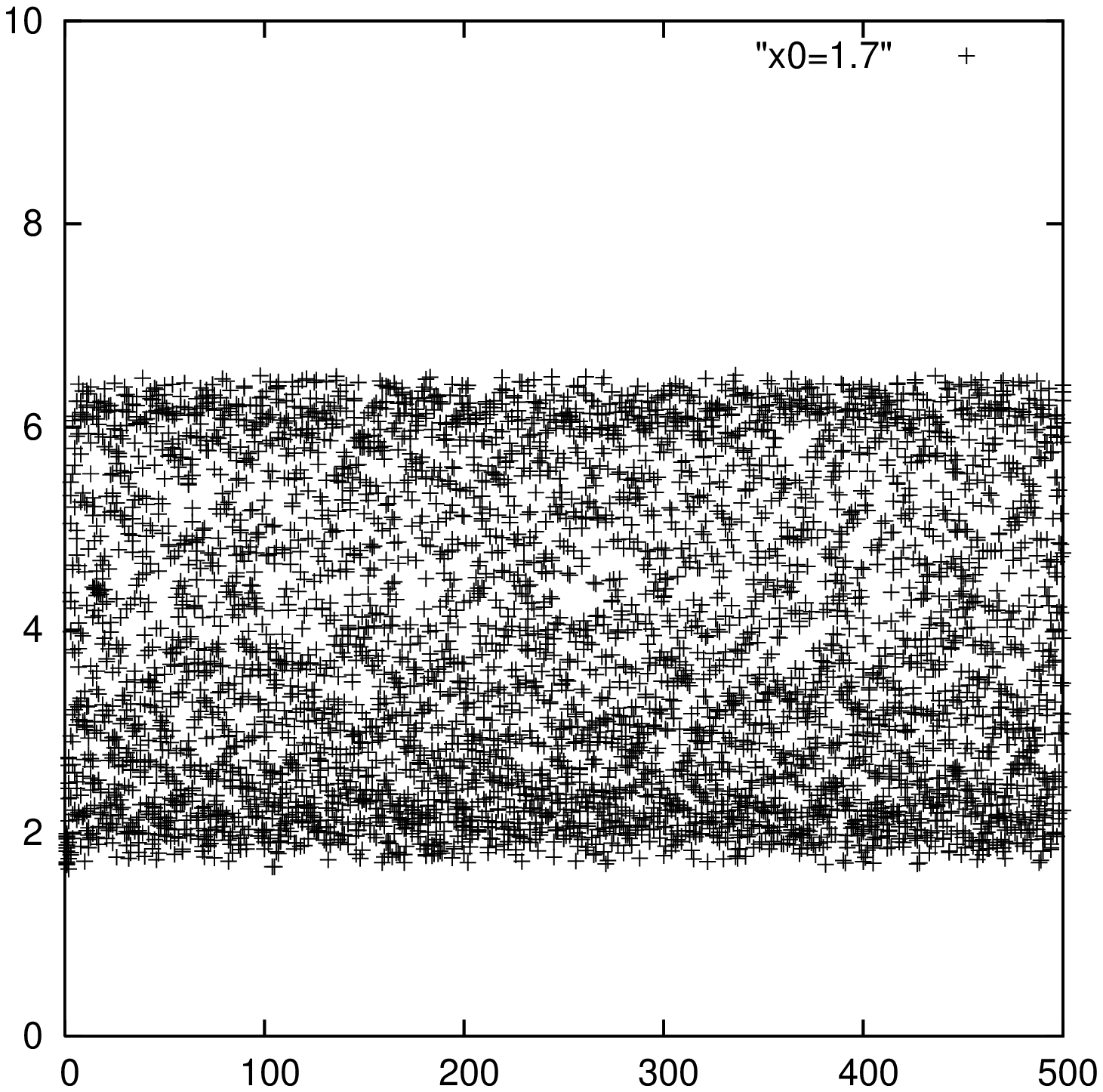}
\caption{Ten runs of (\protect{\ref{eq:main}}) with $f$ as in (\protect{\ref{def:4f}}) for each of $x_0=1.4$ (left), $x_0=1.5,1.6$ 
(middle) and $x_0=1.7$ (right) for 
$l=0.2$.}
\label{figure3}
\end{figure}

For comparison, let us present several simulations for smaller $l=0.05$, see Fig.~\ref{figure4}.

\begin{figure}[ht]
\centering  
\includegraphics[height=.15\textheight]{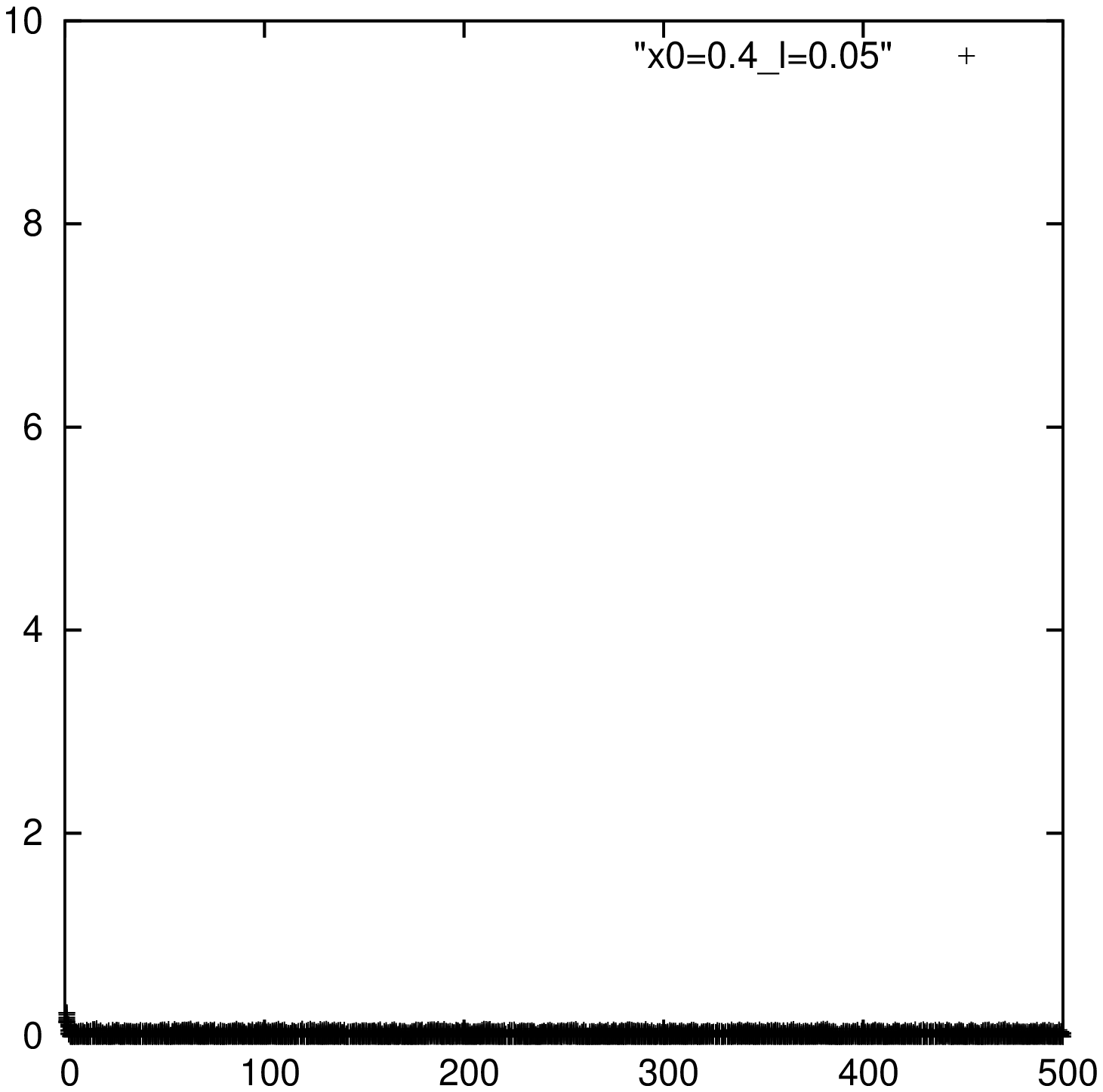}
\hspace{5mm}
\includegraphics[height=.15\textheight]{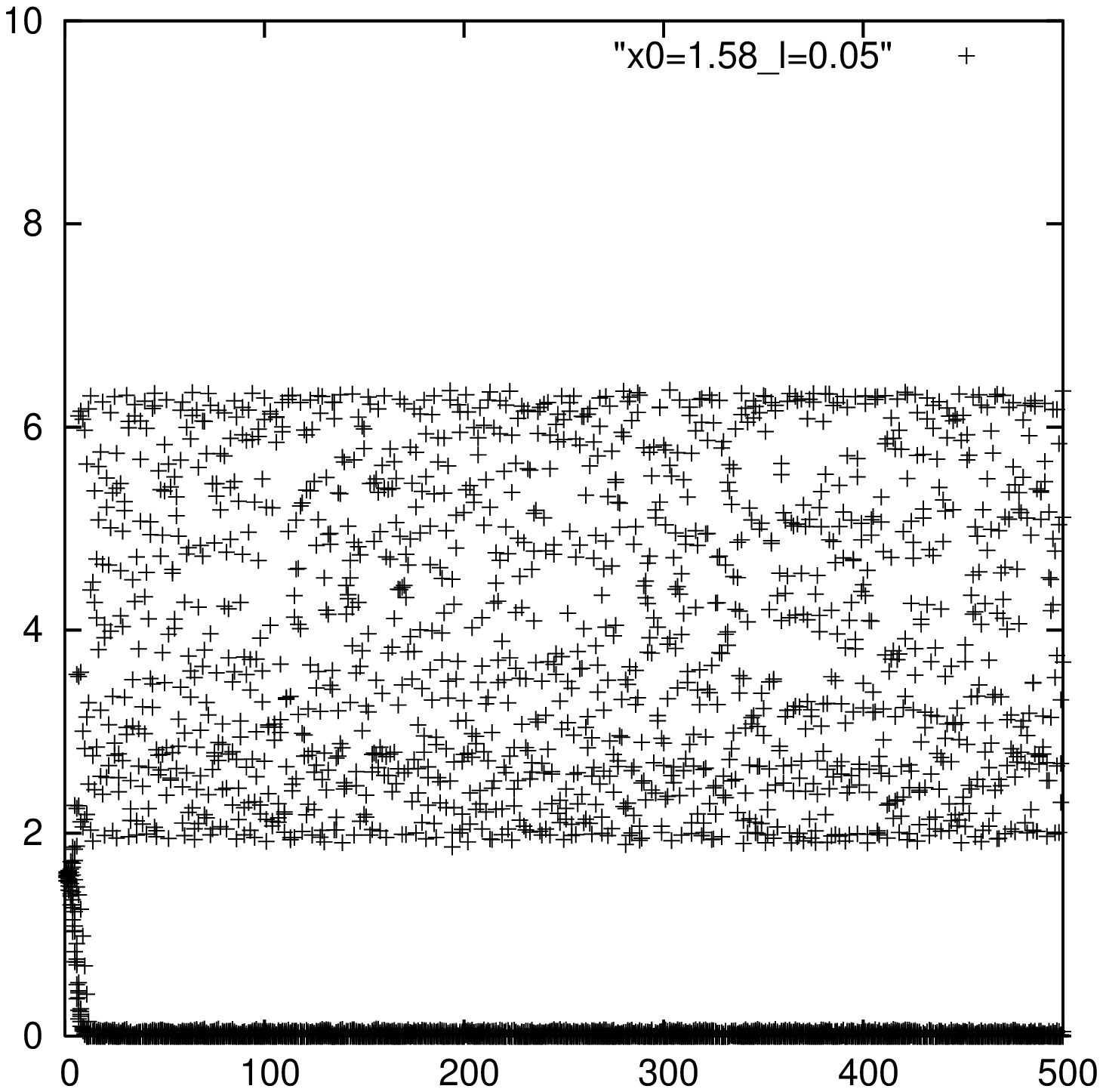}
\hspace{5mm}
\includegraphics[height=.15\textheight]{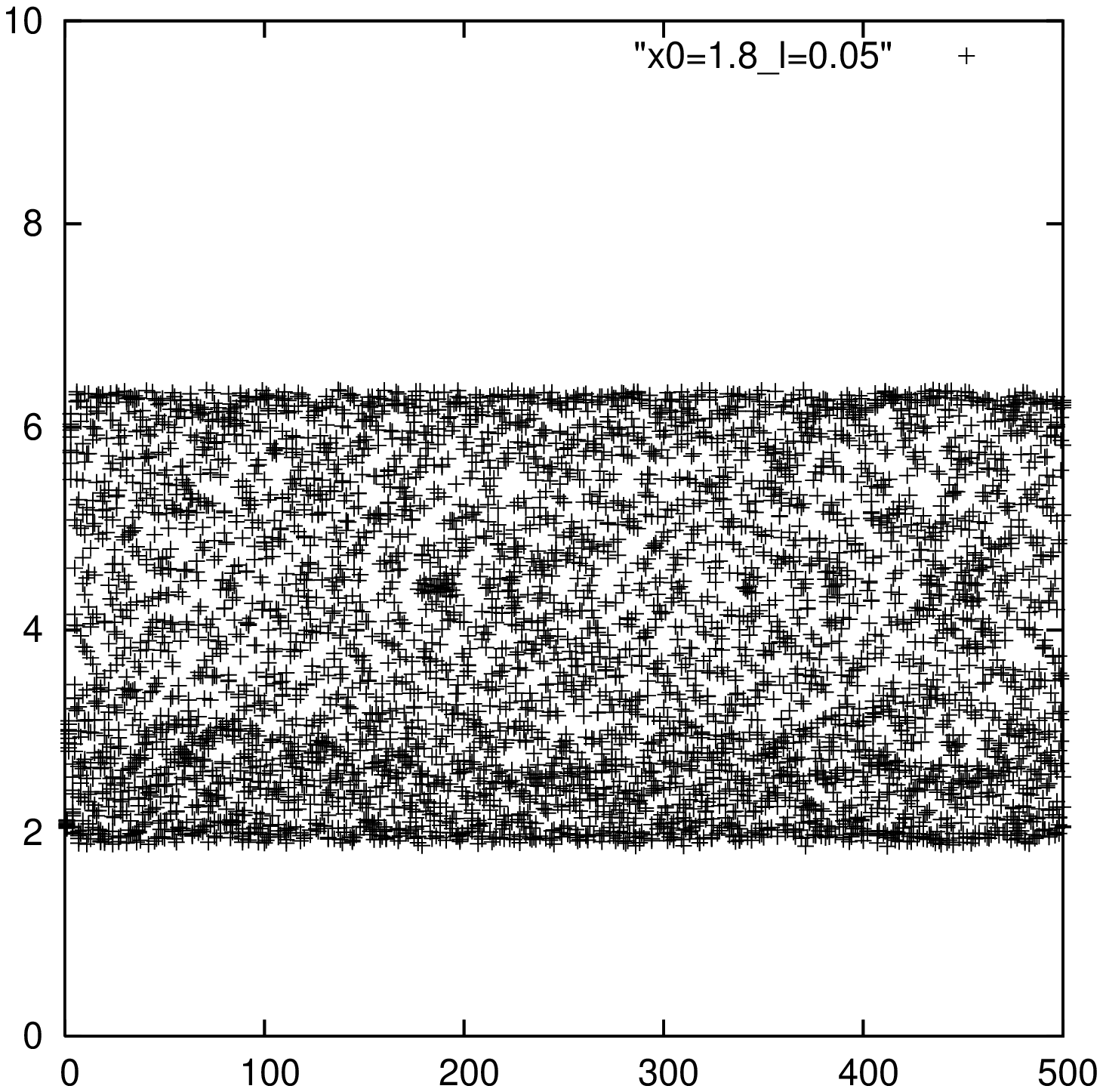}
\caption{Ten runs of (\protect{\ref{eq:main}}) with $f$ as in (\protect{\ref{def:4f}}) for each of $x_0=0.4$ (left), $x_0=1.58$ 
(middle) and $x_0=1.8$ (right) for
$l=0.05$.}
\label{figure4}
\end{figure}
\end{example}

\begin{example}
\label{ex:pers2}
Consider \eqref{eq:main} with
\begin{equation}
\label{def:42f}
f(x):=\frac{4x^2}{2+x}e^{2(1-x)}, \quad x>0.
\end{equation}
The fixed points are $c \approx 0.0833$ and $d\approx 1.2037$, the maximum 
$f_m \approx 1.3688$ is attained at $\approx 0.8508$.
The minimum of $F(x)$ on $[0,c]$ is attained at $b \approx 0.0392$ and equals $F(b) \approx -0.0186$.


Take $a=0.2$, $H=1.8>f_m$,  $f(a) \approx 0.3602$, $F(a) \approx 0.16$, $f(H)\approx 0.6886> f(a)$, 
$-F(H) \approx 1.111$; we can choose $l< 0.16$. If $l \in (-F(b),0.16)$, or $l \in (0.0186,0.16)$, we have 
persistence for any initial condition.
All numerical runs are for the case when $\chi$ is uniformly distributed on $[-1,1]$. 
We observe that for $l>- F(b)$, say, $l=0.04$, we have eventual 
persistence even for small $x_0=0.01$ (Fig.~\ref{figure5}, left) 
while observe Allee effect for smaller $l=0.01<- F(b)$ and the same initial 
value (Fig.~\ref{figure5}, right). This example illustrates the possibility to alleviate the Allee effect with 
large enough random noise. 
Fig.~\ref{figure5} (left) also illustrates the multi-step lifts to get into the persistence area.

\begin{figure}[ht]
\centering
\includegraphics[height=.2\textheight]{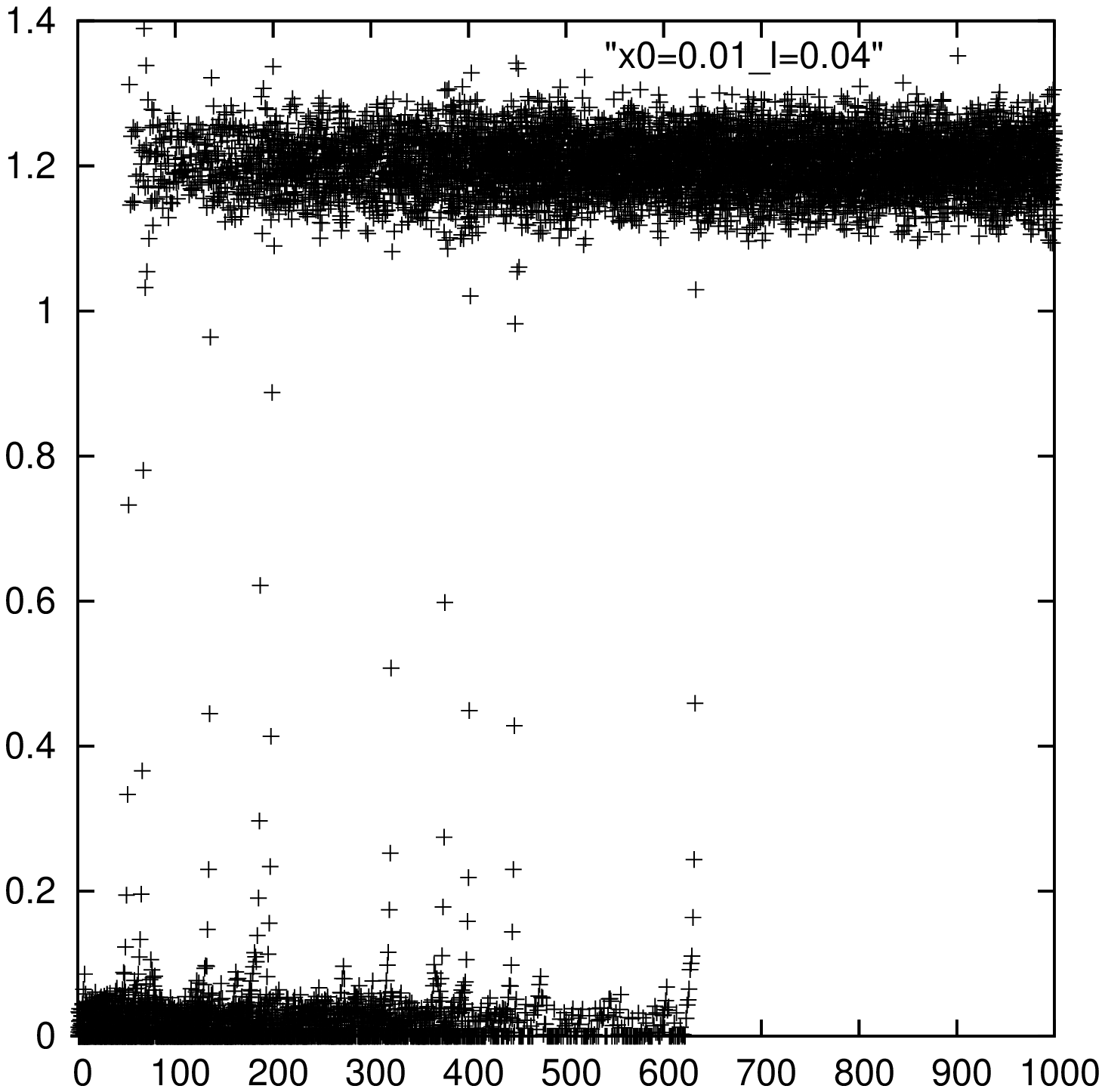}
\hspace{20mm}
\includegraphics[height=.2\textheight]{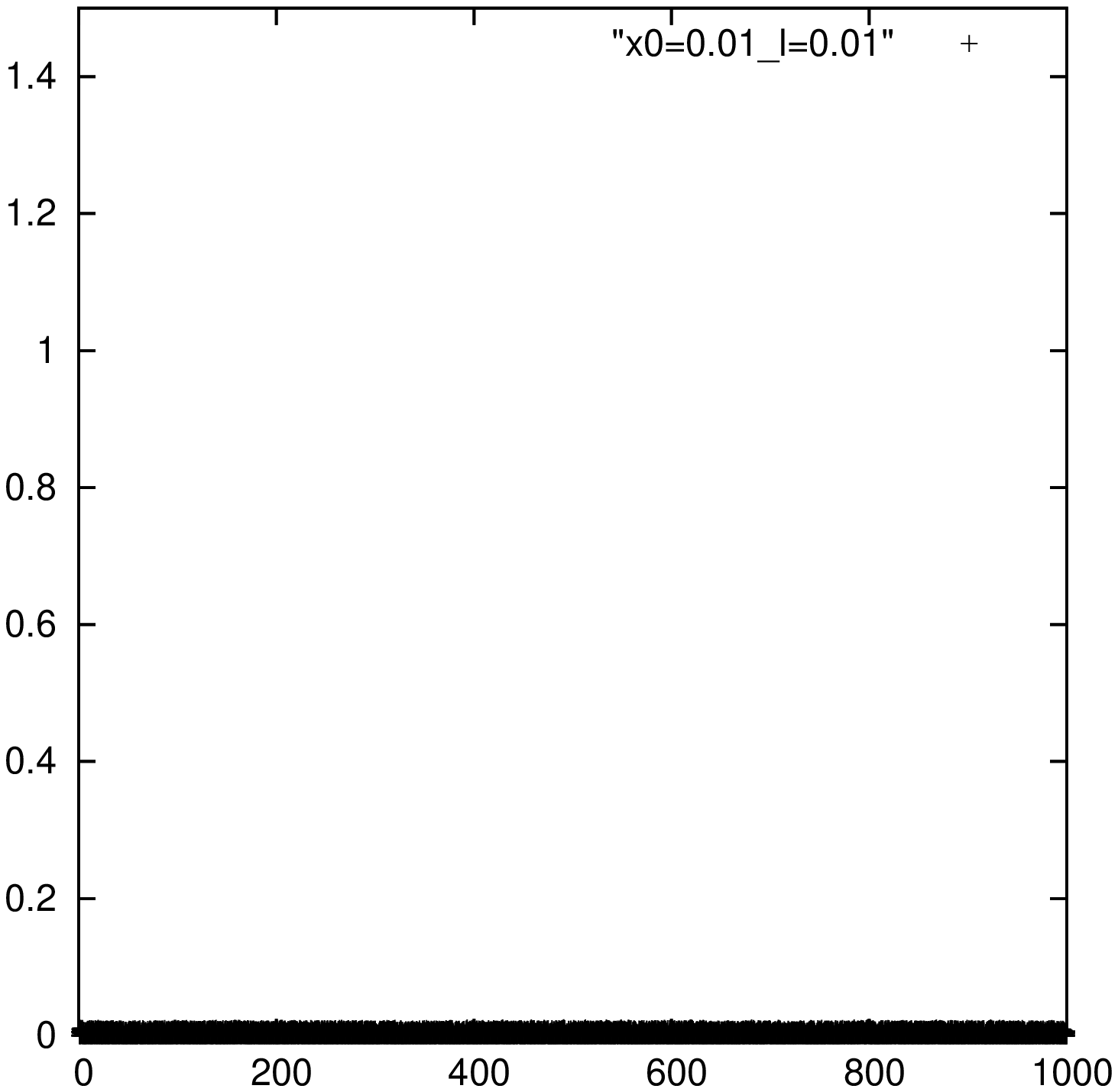}
\caption{Ten runs of (\protect{\ref{eq:main}}) with $f$ as in (\protect{\ref{def:42f}}) for $x_0=0.01$, $l=0.04$ (left), 
and $l=0.01$ (right).}
\label{figure5}
\end{figure}
\end{example}


\section{Discussion}
\label{sec7}

Complicated and chaotic behavior of even simple discrete systems leads to high risk of extinction.
However, frequenly observed persistence suggested that there are some mechanisms for this type of dynamics.
In the present paper, we proposed two mechanisms for sustaining a positive expectation in populations
experiencing the Allee effect:
\begin{enumerate}
\item
By Lemma~\ref{lem:add3}, in the presence of a stochastic perturbation, there is a positive eventual
expectation for any solution, independently of initial conditions. 
This can be treated as persistence thanks to some sustained levels of occasional immigration. 
However, the lower solution bound is still zero,
and even expected solution averages are rather small and matched to this immigration probability distribution.
\item
The second mechanism is more important for sustainability of populations. It assumes that there is a substantial
range of values, where extinction due to either Allee effect or its combination with overpopulation reaction 
is impossible. For example, under contest competition \cite{Brauer} with the remaining population levels sufficient to sustain,
even for initial values in the Allee zone, large enough stochastic perturbations lead to persistence.
Specifically, the amplitude should exceed the maximal population loss in the Allee area, and at the same time
should not endanger the original sustainability area. The result can be viewed as follows:
if there is the Allee effect and sustainable dynamics for a large interval of values, introduction of a 
potentially large enough stochastic perturbation
can lead to persistence, for any initial conditions.
\end{enumerate}

For smaller perturbation amplitudes, there are 3 types of initial values: attracted to low dynamics a.s., a.s. persistent
and those which can demonstrate each type of dynamics with a positive probability.
As illustrated in Section~\ref{sec6}, all three types of dynamics are possible.

In this paper we consider only bounded stochastic perturbations. 
The assumption of boundedness along with the properties of the function $f$ allows to construct a "trap", the
interval $[a, H]$, into which  any solution eventually gets and stays there.

Assume for a moment that in equation \eqref{eq:main}  instead of bounded we have normally distributed~$\chi_n$. 
Applying the approach of the proof of  Theorem~\ref{lem:add2} for bounded stochastic perturbations, we can show that for any 
initial value $x_0>0$, a solution $x_n$ eventually gets into the interval $(a, H)$, a.s. 
However, if $\chi_n$ can take any negative value with nonzero probability, applying the same method, we can show that  
there is a ``sequence" of negative noises with an absolute value exceeding $H$ pushing the solution out of  
the interval $(a, H)$, a.s. Thus, a.s., for any $n_1 \in {\mathbb N}$, there is an $n \geq n_1$ such that $x_n=0$. 
So the conclusions of Lemma 3.2, (ii), and  Theorem 3.3 are no longer valid. 

Note that from the population model's point of view the assumption that the noise is bounded 
is hardly a limitation since in nature there are no unbounded noises. For a normal type of noise,
considering its truncation can be a reasonable approach to the problem.

\section{Acknowledgment}

The research was partially supported by NSERC grants 
RGPIN/261351-2010 and RGPIN-2015-05976 and also by AIM SQuaRE program.
The authors are grateful to the ananymous reviewer whose valuable comments
contributed to the present form of the paper.

\end{document}